\documentclass[11pt]{amsart}
\usepackage[margin=1.35in]{geometry}
\usepackage{amscd,amsmath,amsxtra,amsthm,amssymb,stmaryrd,xr,mathrsfs,mathtools,enumerate,commath, comment}
\usepackage{stmaryrd}
\usepackage{xcolor}
\usepackage{commath}
\usepackage{comment}
\usepackage{tikz-cd}
\usepackage{longtable} 
\usepackage{pdflscape} 
\usepackage{booktabs}
\usepackage{hyperref}
\definecolor{vegasgold}{rgb}{0.77, 0.7, 0.35}
\definecolor{darkgoldenrod}{rgb}{0.72, 0.53, 0.04}
\definecolor{gold(metallic)}{rgb}{0.83, 0.69, 0.22}
\hypersetup{
 colorlinks=true,
 linkcolor=darkgoldenrod,
 filecolor=brown,      
 urlcolor=gold(metallic),
 citecolor=darkgoldenrod,
 }

\usepackage[all,cmtip]{xy}

\DeclareFontFamily{U}{wncy}{}
\DeclareFontShape{U}{wncy}{m}{n}{<->wncyr10}{}
\DeclareSymbolFont{mcy}{U}{wncy}{m}{n}
\DeclareMathSymbol{\Sh}{\mathord}{mcy}{"58}
\usepackage[T2A,T1]{fontenc}
\usepackage[OT2,T1]{fontenc}

\newtheorem{theorem}{Theorem}[section]
\newtheorem{Lemma}[theorem]{Lemma}

\newtheorem*{theorem*}{Theorem}
\newtheorem*{ass*}{Assumption}
\newtheorem{definition}[theorem]{Definition}
\newtheorem{corollary}[theorem]{Corollary}
\newtheorem{remark}[theorem]{Remark}

\newtheorem{conjecture}[theorem]{Conjecture}
\newtheorem{proposition}[theorem]{Proposition}

\newcommand{\fq}{\mathfrak{q}}

\newcommand{\fB}{\mathfrak{B}}

\newcommand{\Z}{\mathbb{Z}}
\newcommand{\p}{\mathfrak{p}}
\newcommand{\Q}{\mathbb{Q}}
\newcommand{\F}{\mathbb{F}}

\newcommand{\cO}{\mathcal{O}}

\newcommand{\fp}{\mathfrak{p}}

\newcommand{\disc}{\mathrm{Disc}}
\newcommand{\rdisc}{\mathrm{disc}}
\newcommand{\ind}{\mathrm{ind}}
\newcommand{\op}[1]{\operatorname{#1}}

\numberwithin{equation}{section}

\begin{document}

\title[On Malle's conjecture for $S_n\times G$]{On Malle's conjecture for the product of symmetric and nilpotent groups}

\author[H.~Mishra]{Hrishabh Mishra}
\address[H.~Mishra]{Chennai Mathematical Institute, H1, SIPCOT IT Park, Kelambakkam, Siruseri, Tamil Nadu 603103, India}
\email{hrishabh@cmi.ac.in}

\author[A.~Ray]{Anwesh Ray}
\address[A.~Ray]{Chennai Mathematical Institute, H1, SIPCOT IT Park, Kelambakkam, Siruseri, Tamil Nadu 603103, India}
\email{anwesh@cmi.ac.in}

\subjclass[2020]{}

\maketitle

\begin{abstract}
 Let $G$ be a finite nilpotent group and $n\in \{3,4, 5\}$. Consider $S_n\times G$ as a subgroup of $S_n\times S_{|G|}\subset S_{n|G|}$, where $G$ embeds into the second factor of $S_n\times S_{|G|}$ via the regular representation. Over any number field $k$, we prove the strong form of Malle's conjecture (cf. \cite[p.316]{malle2002distribution}) for $S_n\times G$ viewed as a subgroup of $S_{n|G|}$. Our result requires that $G$ satisfies some mild conditions. 
\end{abstract}

\section{Introduction}
\subsection{Motivation and historical context}
\par Let $k$ be a number field and let $n\in \Z_{\geq 1}$. Let $\mathcal{G}$ be a transitive subgroup of $S_n$. Given a number field extension $K/k$, denote by $\widetilde{K}$ its Galois closure over $k$. Suppose that $[K:k]=n$. We enumerate the embeddings $\iota_j :K \hookrightarrow \bar{K}$ for $j=1, \dots, n$, and for $\sigma\in \mathcal{G}$, consider the composite \[\iota_{\sigma(j)}: K\xrightarrow{\iota_j} \bar{K}\xrightarrow{\sigma} \bar{K}.\]  This gives rise to a natural permutation representation and we realize $\op{Gal}(\widetilde{K}/k)$ as a transitive subgroup of $S_n$. Set $\op{disc}(K/k)$ to denote the relative discriminant of $K/k$ and take \[\op{Disc}(K):=| \op{Norm}_{k/\Q}(\op{disc}(K/k))|.\] For $X>0$, let $N_k(\mathcal{G}; X)$ be the number of extensions $K/k$ with $[K:k]=n$ for which $\op{Gal}(\widetilde{K}/k)$ is isomorphic to $\mathcal{G}$ as a permutation subgroup of $S_n$, and $\op{Disc}(K)\leq X$. The quantity $N_k(\mathcal{G}; X)$ depends not only on $\mathcal{G}$ and $k$, but also on the embedding of $\mathcal{G}$ in $S_n$, and this is suppressed in our notation. For $g\in S_n$, set \begin{equation}\label{index definition}
    \op{ind}(g):=n-\text{number of orbits for }g.
\end{equation} We set $1=1_{\mathcal{G}}$ to denote the identity element in $\mathcal{G}$ and $\mathcal{G}^\ast:=\mathcal{G}\backslash \{1\}$.
\par Malle made predictions about the asymptotic growth of $N_k(\mathcal{G};X)$ as $X\rightarrow \infty$. We define
$$a(\mathcal{G}):=\left(\op{min}\{\op{ind}(g)\mid g\in \mathcal{G}^\ast\}\right)^{-1}.$$
The absolute Galois group $G_k:=\op{Gal}(\overline{\Q}/k)$ of $k$ acts on the set of conjugacy classes $C(\mathcal{G})$ of $\mathcal{G}$ via the action of $\overline{\Q}$-characters of $\mathcal{G}$. We define
\[
b(k,\mathcal{G}):=\#\{C \in C(\mathcal{G}) : \op{ind}(C)=\op{ind}(\mathcal{G})\}/G_k.
\]
The following conjecture (cf. \cite[Conjecture 1.1]{malle2004distribution}) is referred to as the \emph{strong form of Malle's conjecture}.
\begin{conjecture}\label{malle's}
    We have that
    \[N_k(\mathcal{G};X)\sim c(k, \mathcal{G}) X^{a(\mathcal{G})}(\log X)^{b(k, \mathcal{G})-1},\]
    for some constant $c(k, \mathcal{G})>0$.
\end{conjecture}
We also have the following weak conjecture (cf. \cite[p.316]{malle2002distribution}) which is referred to as the \emph{weak form of Malle's conjecture}
\begin{conjecture}\label{weak malle's}
     For all $\epsilon>0$,
    \[X^{a(\mathcal{G})}\ll_{k, \mathcal{G}} N_{k}(\mathcal{G};X)\ll_{k, \mathcal{G}, \epsilon}X^{a(\mathcal{G})+\epsilon}.\]
\end{conjecture}
The strong form of Malle's conjecture has been shown to be false, and Kl\"uners provided an explicit counterexample, cf. \cite{kluners2005counter}. However, the weak version of the conjecture is still widely believed to be true. The strong form is however known for various families of groups, some of the well-known cases are listed below.
\begin{itemize}
    \item The conjecture has been proven for the abelian groups by Maki \cite{maki1985density} and Wright \cite{wright1989distribution}.
    \item For the groups $S_n$ for $n\leq 5$, the conjecture was settled by Davenport--Heilbronn \cite{davenport1971density} for $n=3$, for $n=4,5$ over $\Q$ by Bhargava \cite{bhargava2005density, bhargava2010density} and for general number fields $k$ by Bhargava--Shankar--Wang \cite{bhargava2015geometryofnumbers}.
    \item For the dihedral group $D_4\subset S_4$ the conjecture was resolved by by Cohen, Diaz Y. Diaz, and Olivier \cite{cohen2002enumerating}.
    \item Let $G$ be a finite nilpotent group, $p$ be the smallest prime divisor of $|G|$. Assume that all elements of order $p$ are central. Consider $G$ sitting in $S_{|G|}$ via the regular representation. The strong form of Malle's conjecture in this context was recently resolved by Koymans and Pagano \cite{koymanspagano2023}. The weak version of the conjecture for all finite nilpotent groups $G$ sitting in $S_{|G|}$ via the regular representation, was previously settled by Kl\"uners and Malle \cite{klunersmalle}.
    \item Wang \cite{Wang2020} proved the strong form of the conjecture for $S_n\times A$, where $n\in \{3, 4,5\}$ and $A$ is a finite abelian group satisfying some mild conditions. Here, $A$ is viewed as a subgroup of $S_{|A|}$ via the regular representation, and $S_n\times A$ embeds naturally in $S_{n|A|}$ via the inclusions 
    \[S_n\times A\hookrightarrow S_n\times S_{|A|}\hookrightarrow S_{n|A|}. \] The conditions on $A$ were subsequently removed due to work of Masri, Thorne, Wei-Lun and Wang \cite{masri2020malle}.
\end{itemize}

\subsection{Main result}
Let $G$ be a finite nilpotent group and $\ell_G$ the smallest prime factor of $|G|$. We view $G$ as a subgroup of $S_{|G|}$ via the regular representation. On the other hand, $S_n$ is in its natural degree $n$ representation (thus contained in $S_n$ via the identity map). We recall that $a(G)= \frac{\ell_G}{(\ell_G-1)|G|}$, $a(S_n)= 1$, and $b(k, S_n)= 1$ for all $n \geq 2$. It is easy to see that $a(S_n\times G)=\frac{1}{|G|}$. Since ${a(S_n)}{n} > {a(G)|G|}$ we deduce that $b(k,S_n\times G)= b(k,S_n)$, cf.  \cite[p.91 l.-3 -- p. 92, l.2]{Wang2020}. We now state our main result. 
\begin{theorem}\label{main theorem}
Let $G$ be a non-trivial finite nilpotent group and $k$ a number field. Consider the regular representation $\op{reg}_G: G\hookrightarrow S_{|G|}$. Here, $S_n\times G$ embeds in $S_{n|G|}$ via natural inclusions 
    \[S_n\times G\xrightarrow{\op{id}\times \op{reg}_G} S_n\times S_{|G|}\hookrightarrow S_{n|G|}. \]Then there exists a constant $c(k,S_n\times G)>0$ such that
\[
N_{k}(S_n\times G; X) \sim c(k,S_n\times G)X^{\frac{1}{|G|}}
\]
in the following cases:
\begin{enumerate}
    \item $n=3$, if $2 \nmid |G|$,
    \item $n=4$, if $2,3 \nmid |G|$,
    \item $n=5$, if $2,3,5 \nmid |G|$.
\end{enumerate}
Thus, Conjecture \ref{malle's} is true in this setting. 
\end{theorem}
This extends the main result of \cite{Wang2020}. The result is proven via a synthesis of the methods of Kl\"uners--Malle \cite{klunersmalle}, Wang \cite{Wang2020} and Koymans--Pagano \cite{koymanspagano2023}. The conditions on $G$ are consistent with those in \cite{Wang2020}.

\subsection{Outline of the proof} We closely follow the method in \cite{Wang2020}. Two complications have to be addressed. First, the method requires strong uniformity estimates for nilpotent groups. We address this in Section \ref{s 3}. It is worth noting that although a weaker uniformity estimate suffices for proving the main theorem, we pursue a sharp bound to attain the best possible uniformity estimate (see Proposition \ref{loc-uni-est}). This result stands as a potentially independent point of interest. Proposition \ref{bound_Y} gives us a crucial ingredient for proving precise asymptotics for $S_n \times G$-extensions. This technical result is the analogue of \cite[p. 117 (5.6)]{Wang2020}, however, it is proven via a different strategy. For more details, see Section \ref{s 4}.

\subsection{Organization}
\par Including the Introduction the paper consists of four sections. In section \ref{s 2}, several preliminary results are proven on the discriminants of the compositum of two number fields. Furthermore, some analytic tools are introduced here. In section \ref{s 3}, using the parameterization for finite nilpotent extensions of $k$ with fixed Galois group $G$ given by Koymans and Pagano \cite{koymanspagano2023} we prove a local uniformity estimate for nilpotent groups (cf. Proposition \ref{loc-uni-est}). The proof of the main theorem is given in section \ref{s 4}.

\subsection{Outlook} It is possible that the inductive methods used in this article could eventually lead to the proof of the strong form of Malle's conjecture in other similar cases. We have been apprised of ongoing efforts by Alberts, Lemke-Oliver, Wang, and Wood, who are expected to develop more generalized methods similar to Wang and ours to establish the strong form of Malle's conjecture for various groups.

\subsection*{Acknowledgments} We would like to thank Peter Koymans and Jiuya Wang for their insightful comments and suggestions. We also extend our sincerest thanks to the anonymous referee for the excellent and timely report.

\section{Preliminary results}\label{s 2}
\par We begin this section by recalling some notation. Throughout, $k$ will be a fixed number field. Set $\Omega_k$ to denote the set of all primes of $k$. For any non-zero ideal $\mathfrak{I}\subset \cO_k$, set $|\mathfrak{I}|:=|\op{Norm}_{k/\Q}(\mathfrak{I})|$. Given any finite extensions $K/k$ we denote by $\rdisc(K/k)$ the relative discriminant ideal in $k$ and set $\disc(K):=|\rdisc(K/k)|$. We use $\fp$ to denote a finite prime in $k$. Take $\rdisc_{\fp}(K/k)$ to be the ideal $\fp^{\op{ord}_{\fp}(\rdisc(K/k))}$ and $\disc_{\fp}(K)$ its absolute norm. We denote the inertia group at a prime $\fp$ by $I_{\fp}(K)$. Let $G$ be a finite nilpotent group and $n\in \{3,4,5\}$ be an integer. Note that $G$ sits inside $S_{|G|}$ via the regular representation. Therefore, we view the product $S_n \times G$ as a subgroup of $S_n\times S_{|G|}\subset S_{n|G|}$. 
\par We prove a number of preliminary results for the product of $S_n\times G$. We note that the case where $G$ is an abelian group is treated in \cite{Wang2020}. Throughout, it shall be assumed that $G$ satisfies the conditions of Theorem \ref{main theorem}. 
\begin{proposition}\label{disc_bound}
    Let $K/k$ and $L/k$ be number field extensions with $n:=[K:k]$ and $m:=[L:k]$. Assume that $[KL : k] = [K : k][L : k]$. Then we have that
\[
\disc(KL) \leq \disc(K)^m \disc(L)^n.
\]
\end{proposition}
\begin{proof}
    The stated result is \cite[Theorem 2.1]{Wang2020}.
\end{proof}
 If $K/k$ and $L/k$ are tamely ramified at $\fp$, then the inertia groups at $\fp$ are cyclic. Let $n:=[K:k]$ and $m:=[L:k]$. Let $g_K \in S_n$ and $g_L \in S_m$ respectively denote generators for these groups. Then, $\rdisc_{\fp}(K)=p^{\ind(g_K)}$ and $\rdisc_{\fp}(L)=p^{\ind(g_L)}$, where $\ind(g)$ denotes the index of $g$ (cf. \eqref{index definition}). The next result is used to determine $\rdisc_{\fp}(KL)$ in terms of $\rdisc_{\fp}(K)$ and $\rdisc_{\fp}(L)$ at the tamely ramified places. Denote by $\widetilde{K}$ (resp. $\widetilde{L}$) the Galois closure of $K$ (resp. $L$). 

\begin{proposition}\label{tr places}
With respect to notation above, assume that $\widetilde{K}\cap \widetilde{L}=k$ and that $p$ is tamely ramified in $\widetilde{K}$ and $\widetilde{L}$. Then the following assertions hold.
\begin{enumerate}
    \item Let $e_K$ and $e_L$ be the ramification indices of $\widetilde{K}$ and $\widetilde{L}$ at $p$ with $(e_K, e_L) = 1$. Then denote a generator of an inertia group of $KL$ at $p$ by $g_{KL}$, we have
\[
\ind(g_{KL}) = \ind(g_{K})\cdot m + \ind(g_L)\cdot n - \ind(g_K) \cdot \ind(g_L).
\]
\item Let the
generator of an inertia group of $K$ at $p$ be $g_K =\prod_{i}c_i$, and the generator of an inertia group of $L$ at $p$ be $g_L =\prod_{j} d_j$. Then the generator $g_{KL}$ of an inertia group of $KL$ at $p$ satisfies
\[
\ind(g_{KL}) = mn - \sum_{i,j}\mathrm{gcd}(|c_i|,|d_j|).
\]
Here $n = [K : k]$ and $m = [L : k]$.
\end{enumerate}
\end{proposition}
\begin{proof}
   The result above is \cite[Theorem 2.2, 2.3]{Wang2020}.
\end{proof}

\begin{Lemma}\label{disc_inv}
    Given $K/k$ and $L/k$ with $\widetilde{K} \cap \widetilde{L} = k$. The local étale algebra of the compositum $(KL)_{\fp}$ at a prime $p$ could be determined by the local étale algebras $(K)_{\fp}$ and $(L)_{\fp}$. In particular, the relative discriminant ideal $\rdisc_{\fp}(KL)$ as an invariant of $(KL)_{\fp}$ could be determined by $(K)_{\fp}$ and $(L)_{\fp}$.
\end{Lemma}
\begin{proof}
    The result is \cite[Theorem 2.4]{Wang2020}.
\end{proof}

It is convenient to now present a group-theoretic version of Proposition \ref{tr places}. This version will be later employed in index computation within the product of appropriate symmetric and nilpotent groups.
\begin{Lemma}\label{index_comp}
    Let $\sigma \in S_n$ and $\tau \in S_m$ with disjoint cycle decomposition $\sigma = \prod_i c_j$ and $\tau = \prod_j d_j$. Suppose that $(|c_i|,|d_j|)= 1$ for all $i,j$ and consider $(\sigma, \tau) \in S_{n}\times S_m \subset S_{nm}$ then
    \[
\ind(\sigma, \tau) = \ind(\sigma)\cdot m+ \ind(\tau)\cdot n- \ind(\sigma) \cdot \ind(\tau).
\]
\end{Lemma}
\begin{proof}
    The result is essentially the same as \cite[Theorem 2.2]{Wang2020}. Since the proof is relatively short, we include it for completeness. Since $(|c_i|,|d_j|) = 1$ for all $i,j$ the element $(c_i,d_j) \in S_{nm}$ is a cycle of length $|c_i||d_j|$ for each $i,j$. These cycles in $S_{nm}$ are disjoint from each other and $(\sigma, \tau) = \prod_{(i,j)}(c_i,d_j)$. Hence, the number of cycles in the decomposition of $(\sigma,\tau)$ is the product of the number of cycles in the decomposition of $\sigma$ and $\tau$. We recall that the number of such cycles is $n-\ind(\sigma)$ for $\sigma \in S_n$ and $m-\ind(\tau)$ for $\tau \in S_m$. Therefore,
    \[
    nm-\ind(\sigma,\tau) = (n-\ind(\sigma))(m-\ind(\tau))
    \]
    and the Lemma follows.
\end{proof}
Let $\ell_1, \dots, \ell_c$ be the primes that divide $|G|$ and denote by $G(\ell_i)$ the $\ell_i$-Sylow subgroup of $G$. Since $G$ is a nilpotent group, we have the product decomposition $G:= \prod_{i=1}^cG(\ell_i)$. Denote by $\ell_G$ the smallest prime factor of $|G|$. Given $\sigma \in S_n$ and $g \in G$ we compute the index of $(\sigma, g) \in S_{n|G|}$ using the Lemma \ref{index_comp}. The method employed by Wang for abelian groups, as detailed in \cite[Lemmas 2.5, 2.6 and 2.7]{Wang2020} generalizes to nilpotent groups. For the sake of completeness, we provide details. In what follows, we set $G^\ast$ to denote $G\backslash \{1_G\}$ and $\op{ind}(G):=\op{min}\{\op{ind}(g)\mid g\in G^\ast\}$. 
\begin{proposition}\label{ind(sigma, g) propn}
    Let $G$ be a finite nilpotent group in its regular representation $G\hookrightarrow S_{|G|}$.
    \begin{enumerate}
        \item\label{c1 ind(sigma, g) propn} When $n=3$, suppose that $2 \nmid |G|$. Then we have that \begin{equation}\label{ind(sigma, g) propn f1}
        \ind((12),g)/|G| > 2 \text{ and } \ind((123),g)/|G| > 1
        \end{equation}
        for all $g \in G$.
        \item\label{c2 ind(sigma, g) propn} For $n=4$ assume that $2,3 \nmid |G|$. Then \begin{equation}\label{ind(sigma, g) propn f2}
        \ind((12),g)/|G| > 2\text{ and } \ind((12)(34),g)/|G| > 1, 
        \end{equation}
        \begin{equation}\label{ind(sigma, g) propn f3}
        \ind((123),g)/|G| > 3\text{ and } \ind((1234),g)/|G| > 2
        \end{equation}
        for all $g \in G$.
        \item\label{c3 ind(sigma, g) propn} When $n=5$ assume that $2,3,5 \nmid |G|$. Then we have that \begin{equation}\label{ind(sigma, g) propn f4}
        \ind(\sigma,g)/|G| \geq 1 + \ind(\sigma)-1/7\end{equation} for all $\sigma \in S_5$ and $g \in G$. Further if $\sigma$ is not conjugate to $(12345)$ in $S_5$ then \begin{equation}\label{ind(sigma, g) propn f5}\ind(\sigma,g)/|G| \geq 12/7 + \ind(\sigma).\end{equation}
    \end{enumerate}
\end{proposition}
\begin{proof} Since $g$ sits inside $S_m$ via the regular representation, we have that \begin{equation}\label{ind g equation}\op{ind}(g)=|G|\left(1-\frac{1}{e(g)}\right)\end{equation} where $e(g)$ denotes the order of $g$. Note that since it is assumed that $\ell_G>2$, and hence $e(g)>2$ for all $g\neq 1_G$. It follows that $\op{ind}(g)>\frac{|G|}{2}$ for all $g\in G^\ast$ and \[\op{ind}(G)=|G|\left(1-\frac{1}{\ell_G}\right).\]
\par We begin with \eqref{c1 ind(sigma, g) propn}. For $(12) \in S_3$, using Lemma \ref{index_comp} we deduce that for all $g \in G$ we have that $\ind((12),g) = |G| + 2\cdot \ind(g)$ as $\ind(12) = 1$. We deduce that \[\ind((12),g)/|G| \geq 1 + 2\cdot\ind(G)/|G| > 2\]
since $\ell_G > 2$.
\par For $(123) \in S_3$, if $3 \nmid |G|$, then again using Lemma \ref{index_comp} we obtain that for all $g \in G$, $\ind((123),g) = 2|G| + \ind(g)$ as $\ind(123) = 2$. Hence the inequality holds in this case. If $3 \mid |G|$, then $G= G(3)\times \widetilde{G}$. Given any $g \in G$ we can write it as $(g(3), \widetilde{g}) \in S_{|G(3)|}\times S_{\widetilde{G}} \subset S_{|G|}$. Further we note that $\ind((123),g) = \ind((123),(g(3),\widetilde{g})) = \ind(((123),g(3)),\widetilde{g})$ in $S_{|G|}$. Since $3 \nmid |\widetilde{G}|$ we can use Lemma \ref{index_comp} to deduce that \[\ind(((123),g(3)),\widetilde{g}) = \ind((123,g(3))\cdot(|\widetilde{G}|-\ind(\widetilde{g}))+3\cdot|G(3)|\cdot\ind(\widetilde{g}).\] Now, since the smallest possible value of $\ind((123,g(3))$ is $2|G(3)|$ when $g(3)$ is identity and smallest possible value for $\ind(\widetilde{g})$ is $\ind(\widetilde{G})$ by definition. Hence, \[
\ind((123,g(3))\cdot(|\widetilde{G}|-\ind(\widetilde{g}))+3\cdot|G(3)|\cdot\ind(\widetilde{g}) \geq 2|G| + |G(3)|\cdot\ind(\widetilde{G}) > |G|.\] This concludes the proof.
\par \textit{(2)} We note that  $\ind(12),$ $\ind(123),$ $\ind(1234),$ and $\ind(12)(34)$ are $1, 2, 3,$ and $2$ respectively. Further $2,3 \nmid |G|$ so we can use Lemma \ref{index_comp} and we deduce that for all $g \in G$, $\ind((12),g) = |G| + 3\cdot\ind(g) > 2|G|$, $\ind((123),g) = 2|G| + 2\cdot\ind(g) > 3|G|$ as $\ind(g) \geq \ind(G) = \frac{\ell_G-1}{\ell_G} |G| \geq \frac{4}{5}|G|$, next $\ind(1234, g) = 3|G| + \ind(g) > 2|G|$ and $\ind((12)(34), g) = 2|G| + 2\cdot \ind(g) > |G|$.
\par \textit{(3)} Using Lemma \ref{index_comp} we conclude that 
\[\begin{split}
    \ind(\sigma, g) = & \ind(\sigma)|G| + 5\cdot\ind(g) - \ind(\sigma)\cdot\ind(g) \\
   \geq  & \ind(\sigma)\cdot|G| + (5-\ind(\sigma))\cdot\ind(g) \geq \ind(\sigma)\cdot |G| + (5-\ind(\sigma))\cdot \ind(G) \\
   \geq  & \ind(\sigma)\cdot |G| + (5-\ind(\sigma))\cdot (\ell_G-1)|G|/\ell_G \\
   \geq & \ind(\sigma)\cdot |G| + (5-\ind(\sigma))\cdot 6|G|/7 \\
   \geq & \op{ind}(\sigma) |G|+6|G|/7,
\end{split} \] since $\ell_G \geq 7$. Dividing both sides of the above inequality by $|G|$, we find that 
\[\frac{\op{ind}(\sigma, g)}{|G|}\geq 1+\op{ind}(\sigma) -\frac{1}{7},\] and this proves \eqref{ind(sigma, g) propn f4}. 
For $\sigma \in S_5$ not conjugate to $(12345)$ we have that $(5-\ind(\sigma))\geq 2$ and hence \[\ind(\sigma)\cdot |G| + (5-\ind(\sigma))\cdot 6|G|/7 \geq \ind(\sigma)\cdot |G| + 12|G|/7,\] from which \eqref{ind(sigma, g) propn f5} follows. This concludes the proof of Proposition.
\end{proof}

\subsection{Bound for product of two groups} 
In this section, we develop upon the results in \cite[Section 3]{Wang2020} in a framework that is appropriate to our applications. Let $S_1$ and $S_2$ be multisets with entries in positive integers. Furthermore, assume that $S_2$ contains some integer other than $1$. We associate the following counting functions with these multi-sets for $i=1,2$
\begin{equation}
    F_i(X):= \#\{s_i \in S_i : s_i \leq X\}.
\end{equation}
We suppose that $F_i(X)$ is finite for all $X>0$ and $i=1,2$. Given positive integers $a$ and $b$, consider the following product counting function,
\begin{equation}
    P_{a,b}(X):= \#\{(s_1,s_2) \in S_1\times S_2 : s_1^as_2^b \leq X\}.
\end{equation}
 The next result shows that given asymptotic bounds for $F_1(X), F_2(X)$ one can obtain results about bounds on the product counting function.
\begin{proposition}\label{upper prod}
   Let $a,b$ be positive integers and $\alpha\in (0,1)$ is a positive real number satisfying the relation\begin{equation}b-a\cdot\alpha > 0.\end{equation} Furthermore, suppose that
    \[
    F_1(X) \leq C_1 X \text{ and }F_2(X) \leq C_2 X^{\alpha}\log^{\beta}(X),
    \]
    for some constants $C_1, C_2,$ and $\beta >0$. Then, we have the following upper-bound
   \begin{equation}
       P_{a,b}(X) \ll_{a, b, \alpha, \beta} C_1C_2X^{1/a}.
   \end{equation}
\end{proposition}
\begin{proof}
The result follows from \cite[Lemma 3.2]{Wang2020}, setting $n_1:=1/a$ and $n_2:=\alpha/b$ (in accordance with the notation in \emph{loc. cit.})). We note that the condition $b-a\alpha>0$ translates to $n_1>n_2$, which is a necessary condition for the Lemma 3.2 of \emph{loc. cit.} to apply. 
\end{proof}
A lower bound for $P_{a,b}(X)$ is given below.
\begin{proposition}\label{lower prod}
    Suppose there exist real constants $c_1, c_2 > 0$ such that
    \[
    c_1 X \leq F_1(X) \text{ and } c_2 X^{\alpha} \leq F_2(X)
    \]
    for some $\alpha \in (0,1)$ and $b-a\cdot\alpha > 0$ then, for all large enough values of $X$, we have that
   \begin{equation}
       c X^{1/a} \leq P_{a,b}(X),
   \end{equation}
   where $c:= \min\{c_1c_2, c_1c_2a/(b-a\cdot\alpha)\cdot(1-2^{\alpha - b/a})\}$.
\end{proposition}
\begin{proof}
    The proof of our result adapts the first part of \cite[Lemma 3.2]{Wang2020}. Given $r \in (0,\infty)$, we set $n_i(r)$ to be the number of times $r$ occurs in the set $S_i$. From the definition of $P_{a,b}(X)$, we deduce that
    \begin{equation}
        \begin{split}
            P_{a,b}(X) &= \sum_{r_1^ar_2^b \leq X} n_1(r_1)n_2(r_2) = \sum_{r^b\leq X} n_2(r)F_1((X/r^b)^{1/a})\\
             &\geq c_1 X^{1/a} \sum_{r^b\leq X} \frac{n_2(r)}{r^{b/a}}.
        \end{split}
    \end{equation}
Using Abel's summation formula,
\[
\sum_{r\leq X^{1/b}} \frac{n_2(r)}{r^{b/a}} \geq \frac{c_2\underline{X}^{\alpha}}{\underline{X}^{b/a}}+c_2\int_1^{\underline{X}}\frac{t^{\alpha}}{t^{b/a+1}}dt,
\]
where $\underline{X}$ denotes the largest integer smaller or equal to $X^{1/b}$ such that $n_2(\underline{X}) > 0$. Noting that the integral $\int_1^{\underline{X}}\frac{t^{\alpha}}{t^{b/a+1}}dt = a/(b-a\cdot\alpha)\cdot(1-\underline{X}^{\alpha - b/a})$. Hence, \[\int_1^{\underline{X}}\frac{t^{\alpha}}{t^{b/a+1}}dt \geq a/(b-a\cdot\alpha)\cdot(1-2^{\alpha - b/a})\] for $\underline{X} \geq 2$. This concludes the proof.
\end{proof}
Let $G$ be an appropriate finite nilpotent group satisfying condition in Theorem \ref{main theorem} for some $n \in \{3,4,5\}$. Consider a property $\mathcal{P}_1$ of degree $n$, $S_n$-extensions of $k$, and a property $\mathcal{P}_2$ of $G$-extensions of $k$ and an ordering of such extensions by some invariants $\op{inv}_1$ and $\op{inv}_2$ valued in positive integers respectively. Assume that the following bounds are satisfied
\begin{equation}\label{upper bound P_1 and P_2}
    \begin{split}
        & \#\{K\in \mathcal{P}_1\mid  \text{ and }\op{inv}_1(K) \leq X\} \ll_{k, S_n} X\\
        & \#\{L\in \mathcal{P}_2\mid  \text{ and }\op{inv}_2(L) \leq X\} \ll_{k, \epsilon} X^{\alpha}\log^{\beta}(X).
    \end{split}
\end{equation}
Then, as an immediate application of previous results by taking the multi-sets $S_1, S_2$ to be the multi-sets consisting of $\op{inv}_1(K)$ and $\op{inv}_2(L)$, as $K$ and $L$ range over $\mathcal{P}_1$ and $\mathcal{P}_2$ respectively. The product counting function,
\begin{equation}
    \begin{split}
       P_{|G|,n, k}(X) =& \#\{(K,L) : [K:k]=n,[L:k]=|G|, \op{Gal}(\widetilde{K}/k) \simeq S_n, \op{Gal}(L/k) \simeq G,\\
        & K\in \mathcal{P}_1, L \in \mathcal{P}_2, \text{ and }\op{inv}_1(K)^{|G|}\op{inv}_2(L)^n \leq X\},
    \end{split}
\end{equation}
satisfies the upper bound,
\[
P_{|G|,n, k}(X) \ll_{k, S_n, G, \epsilon} X^{\frac{1}{|G|}},
\]
whenever $n-|G|\alpha > 0$. Similarly, if we have lower bounds of the form,
\begin{equation}\label{lower bound P_1 and P_2}
    \begin{split}
        X \ll_{k, S_n}\#\{K : [K:k]=n, \op{Gal}(\widetilde{K}/k) \simeq S_n, K\in \mathcal{P}_1, \text{ and }\op{inv}_1(K) \leq X\}\\
       X^{\alpha} \ll_{k,G} \#\{L : [L:k]=|G|, \op{Gal}(L/k) \simeq G, L\in \mathcal{P}_2, \text{ and }\op{inv}_2(L) \leq X\}.
    \end{split}
\end{equation}
Then we deduce that,
\[
 X^{\frac{1}{|G|}} \ll_{k, S_n, G} P_{|G|,n, k}(X),
\]
whenever $n-|G|\alpha > 0$. We take note of these observations below.
\begin{corollary}\label{3.3}
    With respect to the notation above, assume that
    \begin{enumerate}
        \item the bounds \eqref{upper bound P_1 and P_2} and \eqref{lower bound P_1 and P_2} are satisfied for $\alpha:=a(G)$,
        \item $n-|G|a(G)> 0$.
    \end{enumerate} Then we have the following bounds,
    \[
    X^{\frac{1}{|G|}} \ll_{k, S_n, G} P_{|G|,n, k}(X) \ll_{k, S_n, G, \epsilon} X^{\frac{1}{|G|}}.
    \]
\end{corollary}
\begin{proof}
    The result follows from \eqref{upper bound P_1 and P_2} and \eqref{lower bound P_1 and P_2} and Propositions \ref{upper prod} and \ref{lower prod}.
\end{proof}
\section{Local uniformity estimates}\label{s 3}
In this section, we state local uniformity estimates for degree $n$ extensions of $k$ with Galois group $S_n$ given in \cite[Section 4]{Wang2020} for $n=3,4,$ and, $5$. We refer to an extension $K/k$ as a degree $n$, $S_n$-extension if $[K:k]=n$ and $\op{Gal}(\widetilde{K}/k)\simeq S_n$. Furthermore, we use the parameterization for nilpotent extensions developed in \cite{koymanspagano2023} to prove a similar estimate for $G$-extensions of $K$, where $G$ is any finite nilpotent group. We refer to an extension $L/k$ as a $G$-extension if $L$ is Galois over $k$ and $\op{Gal}(L/k) \simeq G$. Let $\fp$ be a finite prime of $k$, and $K/k$ be an $S_4$-extension then, we say that $\p$ is \textit{overramified} in $K/k$ if $\p$ factors into primes in $K/k$ as $\fB^4, \fB^2,$ or $\fB_1^2\fB_2^2$ for finite $p$. For an archimedean $\p$, it must factor into a product of two ramified places.
 
 \par The local uniformity estimates for $S_n$ are stated below.
\begin{theorem}\label{est_sn}
    Let $\fq\subset \cO_k$ be any (non-zero) squarefree ideal, and write $\fq = \prod_{i=1}^f \fp_i$, where $\fp_1, \dots \fp_f$ are distinct prime ideals. Then, the following assertions hold.
    \begin{enumerate}
        \item\label{p1est_sn} Let $M_{3,\fq}(k,X)$ be the number of non-cyclic cubic extensions $K/k$ such that 
        \begin{itemize}
            \item $K/k$ is totally ramified at $\fp_1,\dots, \fp_f$,
            \item $\disc(K) \leq X$.
        \end{itemize}Then, for any $\epsilon>0$, we have that \[M_{3, \fq}(k,X)=O_{\epsilon}(X/|\fq|^{2-\epsilon}).\]
        \item\label{p2est_sn} The number of $S_4$ quartic extensions $K/k$ which are overramified at all primes $\fp_1,\dots, \fp_f$ and $\disc(K) \leq X$ is denoted $M_{4, \fq}(k, X)$. Then, for any $\epsilon>0$, we obtain the bound \[M_{4, \fq}(k, X)=O_{\epsilon}(X/|\fq|^{2-\epsilon}).\]
        \item\label{p3est_sn} The number of $S_5$ quintic extensions $K/k$ which are totally ramified at all primes $\fp_1,\dots, \fp_f$ and $\disc(K) \leq X$ is denoted $M_{5, \fq}(k, X)$ and one has the following bound \[M_{5, \fq}(k, X)=O_{\epsilon}(X/|\fq|^{2/5-\epsilon}).\]
    \end{enumerate}
    In each case, the implied constant is independent of the ideal $q$.
\end{theorem}
\begin{proof}
    The statement \eqref{p1est_sn} is \cite{dw88}[Proposition 6.2] and for the proof of \eqref{p2est_sn},\eqref{p3est_sn} we refer to \cite{Wang2020}[Theorem 4.3, Theorem 1.3] respectively.
\end{proof}
\subsection{A parameterization for nilpotent extensions} Let $G$ be a finite nilpotent group. We recall that $G$ is a product of its Sylow subgroups $G=\prod_{j=1}^c G(\ell_j)$. Fix an algebraic closure $\bar{k}$ of $k$, and set $\op{G}_k$ to denote the absolute Galois group $\op{Gal}(\bar{k}/k)$. In this section, we briefly recall the parameterization of $G$-extensions of $k$ from \cite[section 2]{koymanspagano2023}. This parameterization comes in handy in the proof of the local uniformity estimate for $G$-extensions.

\par We begin by considering the case when $G$ is an $\ell$-group, i.e., $|G|$ is a power of $\ell$. Let $S_\infty\subset \Omega_k$ be the archimedian primes of $k$. We view $k$ is a subfield of $\bar{\Q}$ and set $\op{G}_k$ to denote the Galois group $\op{Gal}(\bar{\Q}/k)$. Likewise, for each prime $\fq\in \Omega_k$, denote by $k_{\fq}$ the completion of $K$ at $q$ and set $\op{G}_{k_{\fq}}:=\op{Gal}(\bar{\Q}_{\fp}/k_{\fq})$. For each prime $q\in \Omega_k$, choose an embedding $\iota_{\fq}: \bar{k}\hookrightarrow \bar{k}_\fq$. The choice of embedding induces an inclusion of $\iota_{\fq}^\ast: \op{G}_{k_{\fq}}\hookrightarrow \op{G}_k$. Denote by $\op{I}_{q}\subset \op{G}_{k_{\fq}}$ the inertia group at $q$. We set $\F_\ell$ to be the field with $\ell$-elements, and set $H^1_{\op{nr}}(\op{G}_{k_{\fq}}, \F_\ell)$ is the image of the inflation map
\[H^1(\op{G}_{k_{\fq}}/\op{I}_{q}, \F_\ell)\xrightarrow{\op{inf}} H^1(\op{G}_{k_{\fq}}, \F_\ell).\]
The classes in $H^1_{\op{nr}}(\op{G}_{k_{\fq}}, \F_\ell)$ are the unramified classes at $\ell$. Let $\mathcal{G}_k^{\op{pro}-\ell}$ denote the maximal pro-$\ell$ quotient of $\op{G}_k$. 

\par Given a subset $S\subset \Omega_k$ that contains $S_\infty$, consider the natural restriction map
\[\Phi(\ell, S):H^1(\op{G}_{k}, \F_\ell)\rightarrow \bigoplus_{q\in \Omega_{k}\setminus S} \frac{H^1(\op{G}_{k_{\fq}}, \F_\ell)}{H^1_{\op{nr}}(\op{G}_{k_{\fq}}, \F_\ell)}.\] There is a finite set of primes $S_{\op{clean}}(\ell)$ containing $S_\infty$ such that the map $\Phi(\ell, S_{\op{clean}}(\ell))$ is surjective, and such that the kernel $\Phi(\ell, S_{\op{clean}}(\ell))$ is finite (cf. \cite[Proposition 2.1 and 2.2]{koymanspagano2023}. Let $\widetilde{\Omega}_k(\ell)$ be the subset of $\Omega_k\setminus S_{\op{clean}}(\ell)$ consisting of primes $q$ with \[\frac{H^1(\op{G}_{k_{\fq}}, \F_\ell)}{H^1_{\op{nr}}(\op{G}_{k_{\fq}}, \F_\ell)} \neq 0.\]
Every prime $\fq \in \Omega_k\setminus \left(S_{\op{clean}}(\ell)\cup\widetilde{\Omega}_k(\ell)\right)$ is unramified in all finite $\ell$-extensions of $K$ (cf. \cite{koymanspagano2023}[Proposition 2.3]). For $\fq\in \widetilde{\Omega}_k(\ell)$ there exists a character $\chi_\fq\in H^1(\op{G}_K, \F_\ell)$ such that $\Phi(\ell, S_{\op{clean}}(\ell))$ has nontrivial coordinate precisely at $\fq$ and at no other prime in $\widetilde{\Omega}_k(\ell)$, see \cite[p. 317, line -9]{koymanspagano2023} for details. The set of characters $\chi_\fq$ (as $\fq$ ranges over $\widetilde{\Omega}_k(\ell)$) is a linearly independent set. There is a positive integer $t$ and a basis 
\[J:=\{\chi_i\mid i=1, \dots, t\}\] of $\op{ker}\left(\Phi(\ell, S_{\op{clean}}(\ell))\right)$ such that 
\[\{\chi_\fq\mid \fq\in \widetilde{\Omega}_k(\ell)\}\cup J\] is a basis of $H^1(\op{G}_K, \F_\ell)$.

\par Let $[0,1]^t$ denote the set of all vectors $\vec{v}=(v_1, \dots, v_t)$ whose entries are either $0$ or $1$. Given $T\in [0, 1]^t$, set $\pi_j(T)$ to denote its $j$-th coordinate. Set $\mathcal{S}_\ell:=[0,1]^{[t]}\times \mathcal{S}_\ell'$, where $\mathcal{S}_\ell'$ is the set consisting of squarefree integral ideals of $\mathcal{O}_k$ supported entirely on $\widetilde{\Omega_k}(\ell)$. Recall that $G^\ast:=G\backslash \{1_G\}$. Two pairs $(T, \mathfrak{b})$ and $(T', \mathfrak{b}')$ in $\mathcal{S}_\ell$ are said to be coprime if $\mathfrak{b}$ and $\mathfrak{b}'$ are coprime and there does not exist a $j\in [t]$, such that $\pi_j(T)=\pi_j(T')=1$. Consider the vectors $(v_g)_{g\in G^\ast}$, indexed by $g\in G^\ast$, having entries $v_g\in \mathcal{S}_\ell$. Denote by $\op{Prim}\left(\mathcal{S}_\ell^{G^\ast}\right)$ the subset of such vectors whose coordinates are pairwise coprime. Elements of $\op{Prim}\left(\mathcal{S}_\ell^{G^\ast}\right)$ shall be represented as $(v_g)_{g\in G^\ast}$, where $v_g=(v_g{(1)}, v_g{(2)})$ with $v_g{(1)}
\in [0,1]^t$ and $v_g{(2)}\in \mathcal{S}_\ell'$. Let $\op{Epi}_{\op{top-gr}}\left(\mathcal{G}_k^{\op{pro}-\ell}; G\right)$ denote the set of surjective homomorphisms $\mathcal{G}_k^{\op{pro}-\ell}\twoheadrightarrow G$.
\begin{proposition}
    There is a surjective map 
    \[P_G: \op{Prim}\left(\mathcal{S}_\ell^{G^\ast}\right)\twoheadrightarrow \op{Epi}_{\op{top-gr}}\left(\mathcal{G}_k^{\op{pro}-\ell}; G\right)\cup \{\cdot\}.\]
\end{proposition}
\begin{proof}
    The above result is \cite[Proposition 2.8]{koymanspagano2023}.
\end{proof}
The pre-image of $\op{Epi}_{\op{top-gr}}\left(\mathcal{G}_k^{\op{pro}-\ell}; G\right)$ is denoted by $ \op{Prim}\left(\mathcal{S}_\ell^{G^\ast}\right)(\op{solv.})$. For an epimorphism $\psi: \mathcal{G}_k^{\op{pro}-\ell}\twoheadrightarrow G$, set $\op{Disc}(\psi)$ to denote the relative discriminant of the corresponding extension of $k$ that is fixed by the kernel of $\psi$. Given an ideal $\mathfrak{b}$, denote by $\op{free}_{S_{\op{clean}}(\ell)}(\mathfrak{b})$ to be the largest ideal that divides $\mathfrak{b}$ and is supported outside $S_{\op{clean}}(\ell)$. 

\begin{proposition}\label{S free discriminant 1}
    Let $v=(v_g(1) , v_g(2))_{g\in G^\ast}$ be an element in $ \op{Prim}\left(\mathcal{S}_\ell^{G^\ast}\right)(\op{solv.})$, then, 
    \[\op{free}_{S_{\op{clean}}(\ell)}\left(\op{Disc}(P_G(v))\right)=\prod_{g\in G^\ast} v_g(2)^{|G| \left(1-\frac{1}{\# \langle g\rangle}\right)}.\]
\end{proposition}
\begin{proof}
The result above is \cite[Proposition 2.10]{koymanspagano2023}.
\end{proof}

\par Next, we consider the more general setting, where $G$ is a nilpotent group. By taking the product of all the maps $P_{G(\ell_j)}$, we get a map
\[P_G: \prod_{j\in [c]} \op{Prim}\left(\mathcal{S}_{\ell_j}^{G(\ell_j)-\op{id}}\right)\twoheadrightarrow \op{Epi}_{\op{top}-\op{gr}}\left(\op{G}_k, G\right)\cup \{\cdot\}.\]We set $S$ to be union $\bigcup_{j\in [c]} S_{\op{clean}}(\ell_j)$, and take $t_j$ to be the value of $t$ for the prime $\ell_j$. Define
\[\mathcal{S}:=\{0,1\}^{[t_1+\dots +t_c]}\times \mathcal{S}',\] where $\mathcal{S}'$ is the set of squarefree ideals supported outside $S$. Given a subset $\Delta\subset G^\ast$, let $\op{Prim}\left(\mathcal{S}^{\Delta}\right)$ be the set of all tuples $(v_g(1), v_g(2))_{g\in \Delta}$ satisfying the following properties
\begin{itemize}
    \item writing $\Pi_i: [\sum t_i]\rightarrow [t_i]$ for the natural projection map, then, $\Pi_i(v_g(1))$ are pairwise coprime,
    \item the $v_g(2)$ are pairwise coprime, 
    \item if $\p$ divides $v_g(2)$ and $\ell$ is a prime that divides the order of $g$, then, 
    \[\# \left(\cO_k/\p\right)\equiv 1\mod{\ell}. \]
\end{itemize}
\par There is a natural bijection 
\[\op{Prim}\left(\mathcal{S}^{G^\ast}\right)\xrightarrow{\sim} \prod_{j\in [c]} \op{Prim}\left(\mathcal{S}_{\ell_j}^{G(\ell_j)-\op{id}}\right), \] cf. \cite[p.326]{koymanspagano2023} for further details. Thus, in the general case we get a map 
\[P_G: \op{Prim}\left(\mathcal{S}^{G^\ast}\right)\twoheadrightarrow \op{Epi}_{\op{top}-\op{gr}}\left(\op{G}_k, G\right)\cup \{\cdot\}\] and set \[\op{Prim}\left(\mathcal{S}^{G^\ast}\right)(\op{solv.}):=P_G^{-1}\left(\op{Epi}_{\op{top}-\op{gr}}\left(\op{G}_k, G\right)\right).\] The next result allows us to read off the $S$-free discriminant of $P_G(v)$ for 
\begin{equation}\label{boring equation 1}v=(v_{g,j}(1), v_{g, j}(2))_{j\in [c], g\in G(\ell_j)-\op{id}}\in \op{Prim}\left(\mathcal{S}^{G^\ast}\right)(\op{solv.}).\end{equation}
\begin{proposition}\label{S clean discriminant main}
    For $v$ as above, we have that 
    \[\op{free}_S\left(\op{Disc}(P_G(v))\right)=\op{free}_S\left(\prod_{G^\ast} v_g(2)^{|G|\left(1-\frac{1}{\#\langle g\rangle}\right)}\right).\]
\end{proposition}
\begin{proof}
    The result is \cite[Proposition 2.13]{koymanspagano2023} and is a direct consequence of Proposition \ref{S free discriminant 1}.
\end{proof}

 Next, we recall some useful estimates from \cite[section 4]{koymanspagano2023}. Let $F$ be a number field and $L$ be a finite abelian extension of $F$. Let $\mathfrak{S}$ be a subset of $\op{Gal}(L/F)$. Denote by $\mathcal{I}_F$ the multiplicative group of non-zero fractional ideals of $F$. For $I\in \mathcal{I}_F$, set $\omega_{\mathfrak{S}}(I)$ to denote the number of prime divisors $\p$ of $I$ such that $\p$ is unramified in $L$, and such that the Frobenius $\op{Frob}_{\p}$ belongs to $\mathfrak{S}$. Let $z$ be a complex number and $\mathcal{P}$ be a finite set of prime ideals of $\cO_F$. Consider the sum
\[A_z(x):=\sum_{\substack{|I|\leq x,\\ \p|I \Rightarrow \op{Frob}_{\p}\in \mathfrak{S}\text{ and }\p\notin \mathcal{P}}} \mu^2(I) z^{\omega_{\mathfrak{S}}(I)},\]

\begin{proposition}\label{}
    Let $F,L,\mathcal{P}$ ans $\mathfrak{S}$ be as above. Then for all positive real numbers $R$, and all $z\in \mathbb{C}$ such that $|z|\leq R$
\begin{equation}\label{koymans-pagano-4.1}
   A_z(x)= C x (\log x)^{\frac{z\#\mathfrak{S}}{\#\op{Gal}(L/F)}-1} + O_{R,F,L,\mathfrak{S},\mathcal{P}}\left(x (\log x)^{\frac{\op{Re}(z)\#\mathfrak{S}}{\#\op{Gal}(L/F)}-2}\right), 
\end{equation}
where $C > 0$ depends only on $z, F, L, \mathfrak{S}, \mathcal{P}$.
\end{proposition}
\begin{proof}
     For the proof of the above estimate we refer to \cite{koymanspagano2023}[Theorem 4.1].
\end{proof}

\subsection{Uniformity estimate for nilpotent groups}
\par Let $G$ be a  non-trivial finite nilpotent group. Given an integral ideal $\fq$ of $\cO_k$, set $N_q(G;X)$ to be the number of Galois extensions $L/k$ with $\op{Gal}(L/k)= G$, $\disc(L)\leq X$, and such that $\fq$ divides $\rdisc(L/k)$. We recall that $\ell_G$ denotes the smallest prime divisor of $|G|$ and $I(G)$ is the subset of elements of $G$ with order $\ell_G$. Set $H(G):= I(G) \cup \{\op{id}\}$. We establish an extension a uniformity estimate for nilpotent groups. This generalizes \cite{Wang2020}[Theorem 4.12]. Recall from \cite{koymanspagano2023} that \[i(G, k) = \frac{\#I(G)}{[k(\mu_{\ell_G}):k]}.\]
\begin{proposition}\label{loc-uni-est}
    Let $G$ be a non-trivial finite nilpotent group and $k$ be a number field. Let $\fq$ be a non-zero ideal of $\cO_k$, and recall that $|\fq|$ denotes the absolute norm to $\Q$. Then, for all $\epsilon>0$, we have the following asymptotic estimate
    \[
    N_q(G;X) \ll_{G,k,\epsilon} \left(\frac{X}{|\fq|^{1-\epsilon}}\right)^{a(G)}(\log X)^{i(G,k)-1}.
    \]
    
\end{proposition}

\begin{proof}
     Without loss of generality, we may assume that $\fq$ is not divisible by any prime in $S$. For $g\in G$, set $e_g := |G|\left(1-\frac{1}{\#\langle g \rangle}\right)$. Observe that for $g\in I(G)$, we have that $e_g=a(G)^{-1}$. It follows from the parameterization in the previous subsection and Proposition \ref{S clean discriminant main} that it suffices to bound the following
    \begin{equation}\label{begin-set}
        \#\left\{v=(v_{g}(1),v_{g}(2))\in \op{Prim}\left(\mathcal{S}^{G^\ast}\right):  \fq \mid \Delta_v \text{ and } |\Delta_v| \leq X \right\},
    \end{equation}
    where \[\Delta_v := \prod_{g \in G-\{\op{id}\}}v_g(2)^{e_g}.\]
     Suppose $(\fq_1, \fq_2)$ is a pair of integral ideals such that $\fq = \fq_1 \fq_2$. We can write 
    \[\begin{split}& v':=\pi_{G-H(G)}(v):=\prod_{g\in G-H(G)} (v_g(1), v_g(2));\\
    & v'':=\pi_{I(G)}(v):=\prod_{h\in I(G)} (v_h(1), v_h(2)).
    \end{split}\]The elements $v'$ and $v''$ belong to $\op{Prim}\left(\mathcal{S}^{G-H(G)}\right)$ and $\op{Prim}\left(\mathcal{S}^{I(G)}\right)$ respectively. We set $N_{\fq_2}(v',X)$ to denote the number of $v''=\prod_{h\in I(G)}(v_h(1), v_h(2))\in \op{Prim}\left(\mathcal{S}^{I(G)}\right)$ such that 
    \begin{itemize}
        \item $ \fq_2 \mid \prod_{h \in I(G)}v_h(2)^{a(G)^{-1}}$,
        \item $|\prod_{h \in I(G)}v_h(2)^{a(G)^{-1}}|\leq \left(\frac{X}{|\prod_{g \in G-H(G)}v_g(2)^{e_g}|}\right)$.
    \end{itemize}
Let $\mathcal{S}\left(\widetilde{\Omega}_K(\ell_G)\right) $ be the set of all squarefree ideals $\mathcal{I}$ supported in $\widetilde{\Omega}_K(\ell_G)$. We note that for $h\in I(G)$, the only prime that divides the order of $h$ is $\ell_G$. All primes $\fp$ that divide $v_h(2)$ must satisfy the congruence
\[\# \left(\cO_K/\fp\right)\equiv 1\mod{\ell_G}. \] Therefore for $h\in I(G)$, the ideal $v_h(2)$ belongs to $\mathcal{S}\left(\widetilde{\Omega}_K(\ell_G)\right)$.
\par Since we prove an upper bound, we may as well neglect the coprimality conditions between $(v_g(1), v_g(2))$ for $g\in G-H(G)$. It suffices to bound the following sum
    \begin{equation}\label{loc}
        \sum_{\substack{v',|\prod_{g \in G-H(G)}v_g(2)^{e_g}|\leq X\\ \fq_1 \mid \prod_{g \in G-H(G)}v_g(2)^{e_g}}}N_{\fq_2}(v',X).
    \end{equation}
    Set $\widetilde{\fq}_2:= \prod_{\fp \mid \fq_2} \fp$ to be the squarefree part of $\fq_2$ and write $\widetilde{\fq}_2 I_2=\prod_{h\in I(G)}v_h(2)$. For $h\in I(G)$, recall that the ideals $v_h(2)$ are all squarefree and mutually coprime. Furthermore, they are supported on $\widetilde{\Omega}_K(\ell_G)$. Therefore, $\widetilde{\fq}_2 I_2$ is squarefree and supported in $\widetilde{\Omega}_K(\ell_G)$. It is easy to see that
    \begin{equation}\label{upper bound N_q}
    N_{\fq_2}(v',X)\leq \sum_{\substack{I_2\\ |\fq_2\left(I_2\right)^{a(G)^{-1}}| \leq \frac{X}{|\prod_{g \in G-H(G)}v_g(2)^{e_g}|},\\ \widetilde{\fq}_2I_2 \text{ supported in }\widetilde{\Omega}_K(\ell_G)}} \mu^2(I_2)\#I(G)^{\omega(\widetilde{\fq}_2)\omega(I_2)},
    \end{equation}
    cf. the proof of \cite{koymanspagano2023}[Theorem 5.1] for further details.

    \par Set $L$ to be the cyclotomic extension $k(\mu_{\ell_G})$ and $\mathfrak{S}=\{1\}\subset \op{Gal}(L/k)$. Note that an ideal $I$ is supported on $\widetilde{\Omega}_K(\ell_G)$ if and only if all prime ideals $\fp|I$ split in $L$. Setting $z:=\# I(G)$, rewrite \eqref{upper bound N_q} as
    \[N_{\fq_2}(v',X)\leq z^{\omega(\fq_2)} A_z\left(Y\right),\]
    where 
    \[Y:=\left(\frac{X^{a(G)}}{|\fq_2|^{a(G)}|\prod_{g \in G-H(G)}v_g(2)^{e_g a(G)}|}\right).\]
    Using the estimate \eqref{koymans-pagano-4.1} we obtain the following bound
\[ N_{\fq_2}(v',X) \ll_{G,k,\epsilon} Y (\log Y)^{\frac{z\#\mathfrak{S}}{[k(\mu_{\ell_G}):k]}-1}.
\]
Noting that \[i(G,k)=\frac{\#I(G)}{[k(\mu_{\ell_G}):k]}=\frac{z\# \mathfrak{S}}{[k(\mu_{\ell_G}):k]},\]
and that 
\[\log(Y)\sim a(G) \log(X),\]we obtain the following upper bound
    
    \begin{equation}\label{N q_2 final bound}
   N_{\fq_2}(v',X) \ll_{G,k,\epsilon} \left(\frac{1}{|\fq_2|}\right)^{a(G)}\frac{X^{a(G)}}{|\prod_{g \in G-H(G)}v_g(2)^{e_g a(G)}|}(\log X)^{i(G,K)-1}.
    \end{equation}

    \par For any $\epsilon>0$, the integral ideal $\fq$ can be written as the product of two proper ideals in at most $O_{\epsilon}(|\fq|^{\epsilon})$ ways.
    From \eqref{N q_2 final bound}, we obtain
    \begin{equation}\label{important bound}\begin{split} & \sum_{\substack{v',|\prod_{g \in G-H(G)}v_g(2)^{e_g}|\leq X\\ \fq_1 \mid \prod_{g \in G-H(G)}v_g(2)^{e_g}}}N_{\fq_2}(v',X) \\
    \ll_{G,k,\epsilon} & |\fq|^{\epsilon/2}\op{max}_{(\fq_1, \fq_2)}\left(\left(\frac{X}{\fq_2}\right)^{a(G)}(\log X)^{i(G,k)-1} \sum_{\substack{v'\\ |\prod_{g \in G-H(G)}v_g(2)^{e_g}|\leq X,\\ \fq_1 \mid \prod_{g \in G-H(G)}v_g(2)^{e_g}}} \frac{1}{|\prod_{g \in G-H(G)}v_g(2)^{e_g a(G)}|}\right).
    \end{split}
    \end{equation}
    Observe that $e_g a(G) > 1$ for $g \in G-H(G)$. Setting $t:=|G|-\#H(G)$, enumerate the elements of $G-H(G)$ as $\{g_1, \dots, g_t\}$. Setting, $e_i:=e(g_i)$, we find that
    \[\sum_{\substack{v'\\ |\prod_{g \in G-H(G)}v_g(2)^{e_g}|\leq X,\\ \fq_1 \mid \prod_{g \in G-H(G)}v_g(2)^{e_g}}} \frac{1}{|\prod_{g \in G-H(G)}v_g(2)^{e_g a(G)}|}\leq \sum_{\substack{I_1, \dots, I_{t}\\ |\prod_i I_i^{e_i}|\leq X\\ \fq_1| \prod_i I_i^{e_i}}} \frac{1}{|\prod_i I_i^{e_i a(G)}|},\]
    where $I_1, \dots, I_t$ are coprime integral ideals. Since $g_i\notin I(G)$, we find that $e_i a(G)>1$. We write $I_i= \fq_{1,i} J_i$, where $J_i$ is coprime to $\fq_1$, and $\fq_1$ divides $\prod_i \fq_{1,i}^{e_i}$. Then, the product $\prod_i I_i ^{e_i a(G)}$ can be written as
    \[\prod_i I_i ^{e_i a(G)}=\prod_i \fq_{1,i} ^{e_i a(G)}\prod_i J_i ^{e_i a(G)}.\]
    Note that the sum 
    \[\sum_{\substack{I_1, \dots, I_{t}\\  \fq_1| \prod_i I_i^{e_i}}} \frac{1}{|\prod_i I_i^{e_i a(G)}|}\] converges and in fact, is bounded above by 
    \[\prod_{i=1}^t\sum_{I} \frac{1}{|I|^{e_i a(G)}}=\prod_{i=1}^t \zeta_{K}(e_i a(G)).\] Let $\fp_1, \dots, \fp_k$ be the prime divisors of $\fq_1$, which factorizes as $\fq_1= \fp_1^{a_1}\dots \fp_k^{a_k}$. Let $\mathfrak{T}$ be the set of all tuples 
    \[\mathfrak{T}=\{(\mathfrak{a}_1, \dots, \mathfrak{a}_t)\mid \mathfrak{a}_i\text{ are coprime, }\prod_i \mathfrak{a}_i=\fp_1\dots \fp_k\}.\]
     Writing $\mathfrak{a}_i=\fp_{i_1}\dots \fp_{i_r}$, set \[\widetilde{\mathfrak{a}}_i:=\prod_{j=1}^r \fp_{i_j}^{\lceil a_{i_j}/e_i\rceil}.\] 
    We write
    \[\begin{split}& \sum_{\substack{I_1, \dots, I_{t}\\  \fq_1| \prod_i I_i^{e_i}}} \frac{1}{|\prod_i I_i^{e_i a(G)}|} \\
    = & \sum_{(\mathfrak{a}_1, \dots, \mathfrak{a}_k)\in \mathfrak{T}}\prod_{i=1}^t \left(\sum_{\substack{I_i \\ \widetilde{\mathfrak{a}}_i|I_i}}\frac{1}{| I_i|^{e_i a(G)}}\right)\\ =& \sum_{(\mathfrak{a}_1, \dots, \mathfrak{a}_k)\in \mathfrak{T}}\prod_{i=1}^t \left(\sum_{J_i}\frac{1}{|\widetilde{\mathfrak{a}}_i|^{e_i a(G)} |J_i|^{e_i a(G)}}\right).\end{split}\]
   We note that $\prod_{i=1}^t \widetilde{\mathfrak{a}_i}^{e_i}$ is divisible by $\fq_1$, and therefore, 
   \[\sum_{\substack{I_1, \dots, I_{t}\\  \fq_1| \prod_i I_i^{e_i}}} \frac{1}{|\prod_i I_i^{e_i a(G)}|} \leq \left(\frac{\# \mathfrak{T}}{|\fq_1|^{a(G)}}\right)\prod_{i=1}^t \left(\sum_{J_i}\frac{1}{|J_i|^{e_i a(G)}}\right)=\left(\frac{\# \mathfrak{T}}{|\fq_1|^{a(G)}}\right)\prod_{i=1}^t \zeta_K(e_i a(G)).\]
   We note that $\#\mathfrak{T}=t^{\omega(\fq_1)}=o(|\fq|^{\epsilon/2})$. 
    Combining the above estimates with \eqref{important bound}, we thus obtain
    \[
    \sum_{\substack{v',|\prod_{g \in G-H(G)}v_g(2)^{e_g}|\leq X\\ \fq_1 \mid \prod_{g \in G-H(G)}v_g(2)^{e_g}}}N_{\fq_2}(v',X) \\
    \ll_{G,K,\epsilon} |\fq|^{\epsilon}\left(\frac{X}{|\fq|}\right)^{a(G)}(\log X)^{i(G,K)-1}.
    \]
    This concludes the proof of the result.
\end{proof}

\section{Proof of the main theorem \ref{main theorem}}\label{s 4}
In this section, $G$ denotes a non-trivial finite nilpotent group. The results in this section are analogous to those in \cite[Section 5]{Wang2020}. We also suppose that $G$ satisfies the appropriate conditions of Theorem \ref{main theorem} for $S_3, S_4,$ or $S_5$. We set 
\[
\begin{split} \mathcal{F}_{k}(S_n\times G; X):=\{ & (K,L) : [K:k]= n, [L:k] = |G|, \op{Gal}(K/k) \simeq S_n, \\ &  \op{Gal}(L/k) \simeq G \text{ and } \disc(KL) \leq X\}
\end{split}
\]
and prove asymptotics for $\# \mathcal{F}_{k}(S_n\times G; X)$ as $X\rightarrow \infty$. We shall prove that the same asymptotics hold for $N_k(S_n\times G; X)$. We first establish a lower bound.
\begin{theorem}\label{lower bound theorem}
    Let $G$ and $k$ be as in Theorem \ref{main theorem}, then, 
    \[  N_{k}(S_n\times G; X) \gg_{G, k} X^{\frac{1}{|G|}}.\]
\end{theorem}

\begin{proof}
Using the upper bound on $\disc(KL)$ from Proposition \ref{disc_bound}, we note that
\begin{align*}
      N_k(S_n\times G; X) \geq & \#\{(K,L) : [K:k]= n, [L:k] = |G|, \op{Gal}(\widetilde{K}/k) \simeq S_n \text{ and }\\
      & \op{Gal}(L/k) \simeq G, \disc(K)^{|G|}\disc(L)^n \leq X\}.
\end{align*}    
Malle's conjecture is known for $S_n$ ($n\in \{3,4,5\})$ \cite{davenport1971density, bhargava2005density, bhargava2010density}. For nilpotent groups, the weak form of the conjecture is known in generality, cf. \cite{klunersmalle}. Thus, we find that
\[\begin{split}
    & N_k(S_n; X)\gg_{n, k} X, \text{ and } \\
    & N_k(G; X) \gg_{G, k} X^{a(G)}.
\end{split}\]
The result then follows from the above bounds and Corollary \ref{3.3} setting $\op{inv}_1(K):=\disc(K)$ and $\op{inv}_1(L):=\disc(L)$.
\end{proof}
\par Next, we examine a set of invariants denoted by $\{\disc_Y\}_{Y>0}$ for $S_n\times G$-extensions, defined below. \begin{definition}\label{def_Y}
    Let $Y > 0$, and suppose $(K,L) \in \mathcal{F}_{k}(S_n\times G; X)$ we define 
    \[
    \disc_{Y}(KL) := \prod_{\fp \in \Omega_k}\disc_{Y, \fp}(KL),
    \]
 where
    \[
\disc_{Y, \fp}(KL) :=
\begin{cases}
    \disc_{\fp}(KL)\text{ }\mathrm{ if }\text{ }|\fp| \leq Y,\\
    \disc_{\fp}(K)^{|G|}\disc_{\fp}(L)^n \text{ } \mathrm{ if } \text{ }|\fp| > Y.
\end{cases}
\]
We set 
\begin{equation}\label{N k y defn}
    \begin{split}N_{k, Y}(S_n\times G; X):= & \# \{(K,L)\mid[K:k]=n, [L:k]=|G|,\\
& \op{Gal}(\widetilde{K}/k)\simeq S_n, \op{Gal}(L/k)\simeq G\text{ and }\disc_Y(KL) < X\}.\end{split}
\end{equation}
\end{definition}
\begin{remark}
    We note here that a similar definition arises in Wang's work. The above definition generalizes \cite[equation (5.4)]{Wang2020}.
\end{remark}

We remark that $\disc(KL) \leq \disc_Y(KL)$ for all $Y>0$ using Proposition \ref{disc_bound}. Hence 
\[N_{k, Y}(S_n\times G; X) \leq N_k(S_n\times G; X)\]
for all $Y> 0$. The invariants $\disc_Y$ approximate $\disc$ as $Y \to \infty$. The counting function associated with $\disc_Y$ is denoted as $N_{k, Y}(S_n\times G; X)$. We establish asymptotics for $N_{k, Y}(S_n\times G; X)$ for $Y > 0$ and leverage these results to derive the intended asymptotics for $N_k(S_n\times G; X)$.
 
 \par Let $S_Y$ be the finite set of all primes of $k$ with absolute norm $|\fp|\leq Y$. Let $\Sigma_Y$ denote the set of tuples $(\mathcal{K}_{\fp}, \mathcal{L}_{\fp})_{\fp \in S_Y}$ of pairs of local étale extensions over $k_{\fp}$ of degree $n$ and $|G|$ respectively. We will think of these tuples as specifying local conditions at all $\fp \in S_Y$ for $S_n$ and $G$-extensions. This argument is indeed an analogue of \cite[p.116, l. -12]{Wang2020}. Accordingly, given a pair $(K,L) \in \mathcal{F}_{k}(S_n\times G; X)$ and a tuple $(\mathcal{K}_{\fp}, \mathcal{L}_{\fp})_{\fp \in S_Y} \in \Sigma_Y$ we write $(K,L) \in ((\mathcal{K}_{\fp}, \mathcal{L}_{\fp}))_{\fp \in S_Y}$, if for each $\fp \in S_Y$,
\begin{enumerate}
    \item the local étale algebra $(K)_{\fp} = \mathcal{K}_{\fp}$, and
    \item the local étale algebra $(L)_{\fp} = \mathcal{L}_{\fp}$.
\end{enumerate}
Next, we have the following Lemma relating $\disc_Y$ and $\disc$ under given local specifications.
\begin{Lemma}\label{pi_ind}
    Let $Y > 0$ and $\Pi :=((\mathcal{K}_{\fp}), \mathcal{L}_{\fp}))_{\fp \in S_Y} \in \Sigma_Y$. If $(K,L) \in \Pi$ then,
    \[
    \disc_Y(KL):=\frac{\disc(K)^{|G|}\disc(L)^n}{d_{\Pi}}
    \]
    where $d_{\Pi}$ depends only on $\Pi$.
\end{Lemma}
\begin{proof}
    The results is similar to Wang's argument (cf. \cite[equation (5.5)]{Wang2020}), which applies in the case when $G$ is abelian. From the definition of $\disc_Y$, it follows that,
    \[
    \disc_Y(KL):= \prod_{\fp \in S_Y}\disc_{\fp}(KL)\prod_{\fp \notin S_Y}\disc_{\fp}(K)^{|G|}\disc_{\fp}(L)^n.
    \]
    Hence, it follows that 
    \[
    \disc_Y(KL):=\frac{\disc(K)^{|G|}\disc(L)^n}{d_{\Pi}}
    \]
    with $d_{\Pi} =  \prod_{\fp \in S_Y}\disc_{\fp}(K)^{|G|}\disc_{\fp}(L)^n\disc_{\fp}(KL)^{-1}$. From Lemma \ref{disc_inv} we deduce that $\disc_{\fp}(K)$,$\disc_{\fp}(L)$ and $\disc_{\fp}(KL)$ are determined by $(K)_{\fp}$ and $(L)_{\fp}$ and hence $d_{\Pi}$ depends only on $\Pi$.
\end{proof}
Before proceeding we state the following easy Lemma from analytic number theory which will be used in the proof of the next Proposition.
\begin{Lemma}\label{ant lemma}
    Let $\{a_n\}$ be a sequence indexed by the positive integers for which $b_n \geq 0$ for all $n\geq 1$. For a fixed positive real number $a>0$ the following are equivalent:
\begin{enumerate}
\item{
for all $\epsilon>0$
\[
\sum_{n\le X} a_n \ll X^{a+\epsilon}\,,
\]
}

\item{
the Dirichlet series $\sum_{n=1}^{\infty} a_n n^{-s}$ converges for all real numbers $s>a$.
}
\end{enumerate}
\end{Lemma}
The above Lemma is also stated in \cite[Lemma 3.1]{MR4047213}. We refer to \cite{tenenbaumbook} for a proof. We state the next Proposition.
\begin{proposition}\label{bound_Y}
    Let $Y > 0$ then we have that,
    \[
    N_{k,Y}(S_n\times G, X) \sim c(k, S_n\times G, Y)X^{\frac{1}{|G|}}.
    \]
\end{proposition}
\begin{proof}
    We count extensions by specifying local conditions at each prime $\fp \in S_Y$. Let $\Pi \in \Sigma_Y$, we consider the counting function
    \begin{align*}
        N_{k,Y,\Pi}(S_n\times G; X) =& \#\{(K,L): [K:k]= n, [L:k] = |G|, \op{Gal}(\widetilde{K}/k) \simeq S_n, \op{Gal}(L/k) \simeq G,\\
         & \text{ and }(K,L) \in \Pi, \disc_Y(KL) \leq X \}.
    \end{align*}
    Using Lemma \ref{pi_ind} we obtain the following equality,
    \begin{align*}
        N_{k,Y,\Pi}(S_n\times G; X) = &\#\{(K,L): [K:k]= n, [L:k] = |G|, \op{Gal}(\widetilde{K}/k) \simeq S_n, \op{Gal}(L/k) \simeq G,\\
         & \text{ and }(K,L) \in \Pi, \disc(K)^{|G|}\disc(L)^n \leq d_{\Pi}X \}.
    \end{align*}
    We use Corollary \ref{3.3} along with the bounds from \cite{bhargava2015geometryofnumbers} and \cite{koymanspagano2023}
    \begin{equation}\label{prev bounds}
    \begin{split}
    & N_k(S_n; X)\ll_{n, k} X, \text{ and } \\
    & N_k(G; X) \ll_{G, k} X^{a(G)}(\log X)^{i(G,k)-1},
\end{split}
\end{equation} to get that
    \[
    N_{k,Y,\Pi}(S_n\times G; X) \ll_{k,S_n,G,\Pi} X^{\frac{1}{|G|}}.
    \]
    Given a $G$-extension $L/k$, we write $L\in \Pi$ to mean that there exists some number field $K$ with $[K:k]=n$ and $\widetilde{K}$ an $S_n$-extension of $k$, for which $(K, L)\in \Pi$. Given any $G$-extension $L \in \Pi$, we also consider the counting function,
    \begin{align*}
        N_{k,Y,\Pi}^L(S_n\times G; X) = &\#\{K: [K:k]= n, \op{Gal}(\widetilde{K}/k) \simeq S_n, (K,L) \in \Pi,\\
        & \text{ and } \disc(K)^{|G|}\disc(L)^n \leq d_{\Pi}X \}.
    \end{align*}
    Since we know precise asymptotics for counting $S_n$ degree $n$ extensions \cite{bhargava2015geometryofnumbers} with finitely many local conditions, we deduce that for any $L \in \Pi$ the asymptotic estimate holds
    \[
    N_{k,Y,\Pi}^L(S_n\times G; X) \sim \frac{c(k, S_n, \Pi)}{\disc(L)^{n/|G|}}\cdot X^{\frac{1}{|G|}}.
    \]
    We also note that,
    \[
    N_{k,Y,\Pi}(S_n\times G; X) = \sum_{L \in \Pi} N_{k,Y,\Pi}^L(S_n\times G; X).
    \]
    Given $Z > 0$, it is clear that,
    \[
    \sum_{\substack{L \in \Pi,\\\disc(L) \leq Z}} N_{k,Y,\Pi}^L(S_n\times G; X) \sim C_Y^Z(k,S_n\times G, \Pi)X^{\frac{1}{|G|}}.
    \]
    Note that $C_Y^Z(k, S_n\times G, \Pi)$ is an increasing function of $Z > 0$ and it is bounded. Hence,
    \[
    \lim_{Z \to \infty}C_Y^Z(k, S_n\times G, \Pi)
    \]
    exists and we denote it by $C_Y(k, S_n\times G, \Pi)$. Clearly, $$C_Y(k, S_n\times G, \Pi) \leq \liminf_{X \to \infty}\frac{N_{k,Y,\Pi}(S_n\times G; X)}{X^{1/|G|}}.$$
    We consider the sum,
    \[
    \sum_{\substack{L \in \Pi,\\\disc(L) > Z}} N_{k,Y,\Pi}^L(S_n\times G; X) \ll_{k,S_n\times G,\Pi, Y} \left(\sum_{\substack{L \in \Pi,\\\disc(L) > Z}} \frac{1}{\disc(L)^{n/|G|}}\right)X^{\frac{1}{|G|}}.
    \]
    Using Lemma \ref{ant lemma} and the bound \ref{prev bounds} we deduce that the sum,
    \[
    \sum_{\substack{L \in \Pi,\\\disc(L) > Z}} \frac{1}{\disc(L)^{s}}
    \]
    converges when $s > a(G)$. Since $\frac{n}{|G|} > a(G)$ we obtain that
    \[
    \lim_{Z \to \infty}\limsup_{X\to \infty} \sum_{\substack{L \in \Pi,\\\disc(L) > Z}} \frac{N_{k,Y,\Pi}^L(S_n\times G; X)}{X^{1/|G|}} = 0.
    \]
    Therefore, we have that,
    \[
    \limsup_{X \to \infty}\frac{N_{k,Y,\Pi}(S_n\times G; X)}{X^{1/|G|}} \leq C_Y(k, S_n\times G, \Pi).
    \]
    Hence, $N_{k,Y,\Pi}(S_n\times G; X)\sim C_Y(k, S_n\times G, \Pi)X^{\frac{1}{|G|}}$. Since the set $\Sigma_Y$ is finite and
    \[
    N_{k,Y}(S_n\times G, X)= \sum_{\Pi \in \Sigma_Y} N_{k,Y,\Pi}(S_n\times G; X),
    \]
    the Proposition follows.
\end{proof}
We have determined precise asymptotics for $N_{k, Y}(S_n\times G; X)$ across all $Y > 0$ mirroring the result we aim to establish for $N_{k}(S_n\times G; X)$. Next, we analyze the difference
\[N_k(S_n\times G; X) - N_{k, Y}(S_n\times G; X)\]
and obtain an upper bound for this expression, thereby concluding the proof of Theorem \ref{main theorem}. Before proceeding, we state the following important observation from index computation in Proposition \ref{index_comp} and local uniformity estimates in Theorem \ref{est_sn}.
Given a square-free ideal $\fq$ in $\cO_k$ and $\sigma \in S_n$ we denote by $N_{k,\fq,\sigma}(S_n; X)$ the counting function for number of $S_n$ degree $n$ extensions $K$ of $k$ such that $\widetilde{K}$ is tamely ramified at all primes $\fp \mid \fq$ with inertia groups equal to $\langle\sigma \rangle$ up to conjugacy. We consider the set,
\[
S_{\sigma}:=\{s \geq 0 : N_{k,\fq,\sigma}(S_n; X) \ll_{\epsilon} X/|\fq|^{s-\epsilon}\text{ for all squarefree }\fq\text{ in }k\},
\]
and we define $s_{\sigma} = \sup S_{\sigma}$. The Theorem \ref{est_sn} gives explicit lower bounds on $s_\sigma$ for certain $\sigma \in C(S_n)$ when $n =3,4,$ or $5$.

\begin{Lemma}\label{main Lemma}
    Let $n \in \{3, 4, 5\}$, let $G$ be a finite nilpotent group. Then,
\[
\ind(\sigma, g)/m - \ind(\sigma) + s_{\sigma} > 1,
\]
for all $(\sigma, g) \in (S_n\times G)\setminus \{(1, 1)\}$.
\end{Lemma}
\begin{proof}
    In the case when $\sigma$ is totally ramified (or overramified), the result follows from Theorem \ref{est_sn}. In all other cases, the result is a direct consequence of Proposition \ref{ind(sigma, g) propn} by simply employing the trivial lower bound $s_\sigma\geq 0$. Note that when $n=5$, and $\sigma$ is conjugate to $(12345)$, then one happens to be in the totally ramified case and Theorem \ref{est_sn} applies.
\end{proof}
\subsubsection{The difference $N_k(S_n\times G; X) - N_{k,Y}(S_n\times G; X)$}
We fix $Y > 0$. From definition \ref{def_Y} it is obvious that this difference is given by,
\begin{equation}
\begin{split}
        & N_k(S_n\times G; X) - N_{k,Y}(S_n\times G; X) \\
        =& \#\{(K,L)\mid  [K:k]= n, [L:k] = |G|, \op{Gal}(\widetilde{K}/k) \simeq S_n,\\ 
        & \op{Gal}(L/k) \simeq G,
         \text{ and } \disc(KL) < X < \disc_Y(KL)\}.
\end{split}
\end{equation}
Our method for bounding the above counting function will closely resemble the proof outlined in Proposition \ref{bound_Y}, wherein we express this counting function as the aggregate of various counting functions parameterized by specific local conditions at various primes in $k$. We proceed by describing these local conditions of our interest, we call such specification of local conditions a \textit{standard tuple}.
\begin{definition}\label{sd tuple}
    A \textit{standard tuple} is a tuple $(S,\mathcal{L})$ satisfying following properties:
    \begin{enumerate}
        \item $S$ is a finite set of primes in $k$,
        \item $\mathcal{L}$ is a tuple $(h_{\fp}, g_{\fp})_{\fp\in S}$ such that:
        \begin{enumerate}
            \item for each $\fp \mid n!|G|$ the pair $(h_{\fp}, g_{\fp})_{\fp\in S}$ is a pair of ramified local étale algebras over $k_{\fp}$ of degree $n$ and $|G|$ respectively,
            \item for each $\fp \nmid n!|G|$ the pair $(h_{\fp}, g_{\fp})_{\fp\in S}$ is a pair of conjugacy classes of inertia generators in $S_n$ and $G$ respectively.
        \end{enumerate}
    \end{enumerate}
\end{definition}
Let $\Sigma_k$ denote the set of all \textit{standard tuples} for the number field $k$. Given $(K,L)$ and a specification of local conditions $\Pi:= (S,\mathcal{L}) \in \Sigma_k$ as above, we write $(K,L) \in \Pi$ if following conditions are satisfied.
\begin{enumerate}
    \item  If $\fp \in \Omega_k$ is such that both $K$, $L$ are ramified at $\fp$ then $\fp \in S$.
    \item For each $\fp \in S$ such that for each $\fp \nmid n!|G|$, we have that $(I_{\fp}(K), I_{\fp}(L)) = (\langle h_{\fp}\rangle,\langle g_{\fp}\rangle)$.
    \item For each $\fp \in S$ such that for each $\fp \mid n!|G|$, we have that $((K)_{\fp}, (L)_{\fp})= (h_{\fp}, g_{\fp})$.
\end{enumerate}
\begin{remark}
    We note that the notion of standard tuples arises in Wang's work; $\Pi$ used in this context is the exact analogue of the set considered in \cite[equation (5.10)]{Wang2020}. 
\end{remark}

Write $K \in \Pi$ if for each $\fp \in S$ we have that, $(K)_{\fp} = h_{\fp}\text{ or }I_{\fp}(K) = \langle h_{\fp}\rangle$ and we write $L \in \Pi$ if for each $\fp \in S$ we have that, $(L)_{\fp} = g_{\fp}\text{ or }I_{\fp}(L) = \langle g_{\fp}\rangle$. Following the notation in \cite{Wang2020}, we write $\exp(h_{\fp}, g_{\fp})$ to denote the exponent of $\disc_{\fp}(KL)$ where $(K)_{\fp} = h_{\fp}$ or $I_{\fp}(K) = \langle h_{\fp}\rangle$ and $(L)_{\fp} = g_{\fp}$ or $I_{\fp}(L) = \langle g_{\fp}\rangle$. The following result shows that $\exp(h_{\fp}, g_{\fp})$ is completely determined by $(h_{\fp}, g_{\fp})$.
\begin{Lemma}\label{exp Lemma}
    In the above setting, $\exp(h_{\fp}, g_{\fp})$ depends only on the tuple $(h_{\fp}, g_{\fp})$.
\end{Lemma}
\begin{proof}
    If $\fp \nmid n!|G|$ then using the \cite[Theorem 2.3]{Wang2020}, the indices $\ind(h_{\fp})$ and $\ind(g_{\fp})$ completely determine $\ind(h_{\fp}, g_{\fp})$ and hence $\exp(h_{\fp}, g_{\fp})$. When $\fp \mid n!|G|$, the Lemma \ref{disc_inv} implies that $\disc_{\fp}(KL)$ is completely determined by the pair $((K)_{\fp}, (L)_{\fp})$ which is the pair $(h_{\fp}, g_{\fp})$.
\end{proof}
As a consequence of the above and our definition of a \textit{standard tuple}, we can prove the following important result.
\begin{Lemma}\label{disc inv 2}
    Let $\Pi:= (S, \mathcal{L}) \in \Sigma_k$ be a standard tuple and suppose that $(K,L) \in \Pi$ then,
    \[
    \disc(KL):=\frac{\disc(K)^{|G|}\disc(L)^n}{d_{\Pi}}
    \]
    where $d_{\Pi}$ depends only on $\Pi$.
\end{Lemma}
\begin{proof}
    This result generalizes \cite[equation (5.11)]{Wang2020}. Let $\fp \notin S$ then either $K/k$ or $L/k$ is not ramified at $\fp$. Without loss of generality, we assume that $K/k$ is not ramified at $\fp$. We deduce from the upper bound in Proposition \ref{disc_bound} that the $\rdisc_{\fp}(KL/k) \mid \rdisc_{\fp}(L/k)^n$. We recall that in the tower $KL/L/k$ we have the following formula,
    \[
    \rdisc(KL/k) = \disc(KL/L)\rdisc(L/k)^n.
    \]
    Therefore, $\rdisc_{\fp}(KL/k) = \rdisc_{\fp}(L/k)^n$. Hence, $\disc_{\fp}(KL) = \disc_{\fp}(L)^n$ or $\disc_{\fp}(KL) = \disc_{\fp}(K)^{|G|}$ for all $\fp \in S$ depending upon whether $\fp$ is unramified in $K$ or $L$. We obtain that,
    \[
    \disc(KL) = \disc(K)^{|G|}\disc(L)^n\prod_{\fp \in S}|\fp|^{\exp(h_{\fp},g_{\fp})-n\cdot\exp(g_{\fp})-|G|\cdot \exp(h_{\fp})}.
    \]
 Lemma \ref{exp Lemma} implies that the quantity
    \[d_{\Pi} = \prod_{\fp \in S}|\fp|^{n\cdot\exp(g_{\fp})+|G|\cdot \exp(h_{\fp})-\exp(h_{\fp},g_{\fp})},\]
    depends only on $\Pi$.
\end{proof}
For $\Pi \in \Sigma_k$, we set
\begin{align}\label{N k Pi}
    N_{k,\Pi}(S_n\times G; X):= \#\{(K,L) \in \mathcal{E}_k: (K,L) \in \Pi\text{ and } \disc(KL) < X \},
\end{align}
where \begin{equation}\label{mathcal{E}_k}\mathcal{E}_k:= \{(K,L): [K:k]= n, [L:k] = |G|, \op{Gal}(\widetilde{K}/k) \simeq S_n\text{ and } \op{Gal}(L/k) \simeq G\}.\end{equation} From Lemma \ref{disc inv 2} we deduce that
\begin{equation}\label{imp eqn}
    \begin{split}
         N_{k,\Pi}(S_n\times G; X)=& \#\{(K,L) \in \mathcal{E}_k: (K,L) \in \Pi \text{ and }\\
          & \disc(K)^{|G|}\disc(L)^n < X d_{\Pi} \}\\
         =& \#\{(K,L)\in \mathcal{E}_k: (K,L) \in \Pi \text{ and }\\
         & \prod_{\fp\notin S}\disc_{\fp}(K)^{|G|}\disc_{\fp}(L)^n < {X}/{\prod_{\fp\in S}|\fp|^{\exp(h_{\fp},g_{\fp})}}\}
\end{split}
\end{equation}

We also consider the same counting function with invariants $\disc_Y$ for $Y > 0$. We denote these by $N_{k,Y,\Pi}(S_n\times G; X)$. Then we have the following expression for any $Y > 0$,
\begin{equation}\label{decomp_diff}
    \begin{split}
        N_{k}(S_n\times G; X)-N_{k, Y}&(S_n\times G; X) \\
         &= \sum_{\Pi \in \Sigma_k} N_{k,\Pi}(S_n\times G; X)-N_{k,Y,\Pi}(S_n\times G; X).
    \end{split}
\end{equation}
Suppose that $(K,L)\in \Pi$. We recall from the conditons on $S$ that any prime $\fp\notin S$ must be unramified in either $K$ or $L$. Therefore, we find that for $\fp\notin S$, 
\[\op{Disc}_{\fp}(KL)=\op{Disc}_{\fp}(K)^{|G|} \op{Disc}_{\fp}(L)^{n}. \]
Thus, if $S \subset S_Y$, then we have that \begin{equation}\label{disceqn}\disc(KL)=\disc_Y(KL)\end{equation} for all $(K,L) \in \Pi$.
\begin{Lemma}
    Let $\Pi\in \Sigma_k$ be such that $\prod_{\fp\in S}|\fp| \leq Y$. Then, \[ N_{k,\Pi}(S_n\times G; X)= N_{k,Y,\Pi}(S_n\times G; X).\]
\end{Lemma}
\begin{proof}
 The condition that $\prod_{\fp\in S} |\fp|\leq Y$ implies that $|\fp|\leq Y$ for all $\fp\in S$. Consequently, we get that $S\subseteq S_Y$, and the conclusion follows from \eqref{disceqn}.
\end{proof}
Therefore, only $\Pi \in \Sigma_k$ with $\prod_{\fp\in S}|\fp| > Y$ have a positive contribution in the above sum \eqref{imp eqn}. This motivates the following definition.
\begin{definition}\label{res disc}
    Let $(K,L) \in \mathcal{E}_k$ (cf. \eqref{mathcal{E}_k}) and $\Pi=(S, \mathcal{L})\in \Sigma_k$. Assume that $(K, L) \in \Pi$. Then, the \emph{restricted discriminant} of $K/k$ is defined as
    \[
    \disc_{res}^{\Pi}(K):= \prod_{\fp \notin S}\disc_{\fp}(K).
    \]
    Define the restricted discriminant $\disc_{res}^{\Pi}(L)$ for $L/k$ in the same way.
\end{definition}
We now prove asymptotic upper bounds for the number of $S_n$ degree $n$ extensions and $G$-extensions when ordered by \textit{restricted discriminant}. We now define the relevent counting functions. Given $\Pi \in \Sigma_k$, we set
\[
N_{k,\Pi,res}(S_n; X):= \#\{K: [K:k]=n, \op{Gal}(K/k)\simeq S_n, K \in \Pi, \text{ and }\disc_{res}^\Pi(K) \leq X\},
\]
and
\[
N_{k,\Pi,res}(G; X):= \#\{L: [L:k]=|G|, \op{Gal}(L/k)\simeq G, L \in \Pi, \text{ and }\disc_{res}^\Pi(L) \leq X\}.
\]
Set $C(S_n)$ to be the set of conjugacy classes in $S_n$. Given $\sigma\in C(S_n)$, we take $\fq_\sigma := \prod_{\fp \in S, h_{\fp} = \sigma}\fp$. Recall that $r_\sigma$ is defined in Lemma \ref{main Lemma}.
\begin{Lemma}\label{res sn}
    With respect to above notation, we have that
    \[
    N_{k,\Pi,res}(S_n; X) \ll_{k,S_n, \epsilon} \left(\prod_{\sigma \in C(S_n)}|\fq_\sigma|^{-r_\sigma + \op{ind}(\sigma)+\epsilon}\right)X.
    \]
\end{Lemma}
\begin{proof}
   This result is adaptation of \cite[equation (5.13)]{Wang2020}. Using definition of $N_{k,\Pi,res}(S_n; X)$ we note that,
    \begin{equation}
        \begin{split}\label{N_{k,Pi,res}(S_n; X)}
            N_{k,\Pi,res}(S_n; X) &\leq  \#\{K: [K:k]=n, \op{Gal}(K/k)\simeq S_n, K \in \Pi, \text{ and }\\
            & \disc(K) \leq X\prod_{\fp\in S}|\fp|^{\exp(h_{\fp})}\}\\
            & \leq \#\{K: [K:k]=n, \op{Gal}(K/k)\simeq S_n, K \in \Pi, \text{ and }\\
            & \disc(K) \leq CX \prod_{\sigma \in C(S_n)}|\fq_\sigma|^{\ind(\sigma)}\},
        \end{split}
    \end{equation}
    
    where $C > 0$ is an absolute constant. This inequality holds because there are only finitely many primes that can be wildly ramified in $K/k$, and there are finitely many local étale algebras over $k_{\fp}$ with a fixed degree for each prime $\fp$ in $k$ and $\exp(h_{\fp}, g_{\fp})$ can be uniformly bounded for all such primes. By recalling the definition of $r_{\sigma}$ and applying Theorem \ref{est_sn} to \eqref{N_{k,Pi,res}(S_n; X)}, with 
    \[\fq=\prod_{\sigma \in C(S_n), r_\sigma \neq 0}\fq_\sigma\] and counting extensions with $\disc$ upto $CX \prod_{\sigma \in C(S_n)}|\fq_\sigma|^{\ind(\sigma)}$, we deduce that
    \[
    N_{k,\Pi,res}(S_n; X) \ll_{k,S_n, \epsilon} \left(\prod_{\sigma \in C(S_n)}|\fq_\sigma|^{-r_\sigma + \ind(\sigma) + \epsilon}\right)X.
    \]
    This concludes the proof of the Lemma.
\end{proof}

Next, we use the local uniformity estimates for $G$-extensions to obtain the asymptotic upper bound for $N_{k,\Pi, res}(G; X)$.
\begin{Lemma}\label{res G}
    The following asymptotic bound holds
    \[
    N_{k,\Pi,res}(G; X) \ll_{k, G, \epsilon} \left(\prod_{\fp \in S}|\fp|^{A\epsilon}\right)X^{a(G)}(\log X)^{i(G,k)-1}.
    \]
    for some constant $A > 0$ depending only on $G$ and $k$. 
\end{Lemma}
\begin{proof}
    Our result is proven as a generalization of \cite[equation (5.14)]{Wang2020}. The proof is a simple application of Proposition \ref{loc-uni-est}. Note that the following inequality holds,
    \begin{align*}
       N_{k,\Pi,res}(G; X) \leq \#\{L: [L:k]=|G|, \op{Gal}(L/k)\simeq G, \text{ and }\disc(L) \leq X\prod_{\fp \in S}|\fp|^{\exp(g_{\fp})}\}.
    \end{align*}
    Applying Proposition \ref{loc-uni-est} with $\fq = \prod_{\fp \in S}\fp^{\exp(g_{\fp})}$ and counting extensions with $\disc$ upto $X\prod_{\fp \in S}|\fp|^{\exp(g_{\fp})}$ we deduce the following bound for the quantity on the right-hand side above and same bound holds for $N_{k,\Pi, res}(G; X)$,
    \[
     N_{k,\Pi,res}(G; X) \ll_{k,G,\epsilon} \left(\frac{X|\fq|}{|\fq|^{1-\epsilon}}\right)^{a(G)}(\log (X|\fq|)^{i(G,k)-1}
    \]
    and the Lemma follows by using the fact that $\log(X) \ll_{\epsilon} X^{\epsilon}$ for some $A>0$. For instance, we can take $A$ to be the sum $a(G)+i(g,k)-1$.
\end{proof}
 In the result that follows, set \[\delta := \max_{\sigma \in C(S_n), g \in G}(-r_\sigma+\ind(\sigma)-\ind(\sigma,g)/|G|).\] Lemma \ref{main Lemma} implies that $\delta<-1$.
\begin{Lemma}\label{prod res bound}
    Given $\Pi \in \Sigma_k$ and $\fq_\sigma$ as in Lemma \ref{res sn}. We have the following asymptotic upper bound 
    \[
    N_{k,\Pi}(S_n\times G, X) \ll_{k,S_n,G, \epsilon} \left(\prod_{\sigma \in C(S_n)}|\fq_\sigma|^{\delta+\epsilon}\right)X^{\frac{1}{|G|}}.
    \]
   
\end{Lemma}
\begin{proof}
Considering the definitions of $N_{k,\Pi}(S_n\times G; X)$ and the \textit{restricted discriminant}, we have:
    \begin{align*}
        N_{k,\Pi}(S_n\times G; X)=& \#\{(K,L)\in \mathcal{E}_k: (K,L) \in \Pi \text{ and }\\
         & \prod_{\fp\notin S}\disc_{\fp}(K)^{|G|}\disc_{\fp}(L)^n < {X}/{\prod_{\fp\in S}|\fp|^{\exp(h_{\fp},g_{\fp})}}\}\\
         =& \#\{(K,L)\in \mathcal{E}_k: (K,L) \in \Pi \text{ and }\\
         & \disc_{res}^{\Pi}(K)^{|G|}\disc_{res}^\Pi(L)^n < {X}/{\prod_{\fp\in S}|\fp|^{\exp(h_{\fp},g_{\fp})}}\}.
    \end{align*}
    We have previously established upper bounds on the number of suitable $S_n$ and $G$-extensions when ordered by the invariant $\disc_{res}$. Therefore, a straightforward application of Corollary \ref{3.3} with $\mathcal{P}_1, \mathcal{P}_2$ properties that $K \in \Pi$, $L \in \Pi$ respectively and $\op{inv}_1, \op{inv}_2$ are \textit{restricted discriminant} together with Lemmas \ref{res sn} and \ref{res G} yields the desired bound on $N_{k,\Pi}(S_n\times G; X)$.
\end{proof}

\par We now give the proof of our main theorem. 

\begin{proof}[Proof of Theorem \ref{main theorem}]
    \par First, we recall that Theorem \ref{lower bound theorem} asserts that 
     \[  N_{k}(S_n\times G; X) \gg_{G, k} X^{\frac{1}{|G|}}.\] Recall that for $Y>0$, the definition of $N_{k,Y}(S_n\times G; X)$ is given in \eqref{N k y defn}. Proposition \ref{bound_Y} asserts that 
 \[
    N_{k,Y}(S_n\times G, X) \sim c(k, S_n\times G, Y)X^{\frac{1}{|G|}}.
    \]
    We remark that $c(k, S_n\times G, Y)$ is monotonically increasing as a function of $Y$. We note that,
    \begin{equation}\label{final eq}
        \lim_{X \to \infty}\frac{N_{k,Y}(S_n\times G, X)}{X^{\frac{1}{|G|}}} \leq \liminf_{X \to \infty}\frac{N_{k}(S_n\times G, X)}{X^{\frac{1}{|G|}}}.
    \end{equation}
    Hence, $c(k, S_n\times G, Y) \leq \liminf_{X \to \infty}\frac{N_{k}(S_n\times G, X)}{X^{1/|G|}}$ for all $Y > 0$ and therefore $$\lim_{Y \to \infty}c(k, S_n\times G, Y)$$ exists and we denote this limit by $c(k, S_n\times G)$.
Next, we consider the difference, 
\[N_{k}(S_n\times G; X)-N_{k, Y}(S_n\times G; X).\]

    From \eqref{decomp_diff} we deduce that
\begin{equation}\label{decomp_diff 2}
    \begin{split}
        N_{k}(S_n\times G; X)-N_{k, Y}(S_n\times G; X) \leq \sum_{\substack{\Pi \in \Sigma_k,\\ \prod_{\fp \in S}|\fp| > Y}} N_{k,\Pi}(S_n\times G, X),
    \end{split}
\end{equation}
(see \eqref{N k Pi} for the definition of $N_{k,\Pi}(S_n\times G, X)$). Let $\Pi \in \Sigma_k$ be a \textit{standard tuple}. Set $\mathcal{T}(\Pi)$ to be the associated tuple $(\fq_\sigma)_{\sigma \in C(S_n)}$, where $C(S_n)$ denotes the set of conjugacy classes of $S_n$ and \[\fq_\sigma := \prod_{\fp \in S, h_{\fp} = \sigma}\fp.\]We claim that for any tuple of ideals $(\fq_\sigma)_{\sigma \in C(S_n)}$ there are at most $O_{\epsilon}(\prod_{\sigma \in C(S_n)}\fq_\sigma)^{\epsilon} $ \textit{standard tuples} $\Pi\in \Sigma_k$ such that
\begin{equation}\label{fib}
    \mathcal{T}(\Pi)=(\fq_\sigma)_{\sigma \in C(S_n)}.
\end{equation}In greater detail, $\Pi$ is determined by the set of primes $S$ and the tuple $\mathcal{L}$ of pairs $(h_{\fp}, g_{\fp})$ for each $\fp \in S$. The equation \eqref{fib} implies that there are at most $2^{\omega(\prod_\sigma \fq_\sigma)}$ choices for the set $S$ and at most $C (n!|G|)^{\omega(\prod_\sigma \fq_\sigma)}$ choices for the tuple $\mathcal{L}$ for some absolute constant $C > 0$. Hence, we deduce that at most $O_{\epsilon}(\prod_{\sigma \in C(S_n)}\fq_\sigma)^{\epsilon} $ choices for $\Pi$.
\par Combining \eqref{decomp_diff 2} with Lemma \ref{prod res bound}, we find that
\begin{equation*}
    \begin{split}
        N_{k}(S_n\times G; X)-N_{k, Y}(S_n\times G; X) &\leq \sum_{\substack{\Pi \in \Sigma_k,\\ \prod_{\fp \in S}|\fp| > Y}} N_{k,\Pi}(S_n\times G, X)\\
         &\ll_{k,S_n,G,\epsilon} \sum_{\substack{\Pi \in \Sigma_k,\\ \prod_{\fp \in S}|\fp| > Y}} \left(\prod_{\sigma}|\fq_\sigma|^{\delta+\epsilon}\right)X^{\frac{1}{|G|}}\\
         &\ll_{k,S_n,G,\epsilon} \left(\sum_{(\fq_\sigma), \prod_{\sigma}|\fq|_\sigma > Y}\prod_{\sigma}|\fq_\sigma|^{\delta+\epsilon}\right) X^{\frac{1}{|G|}}\\
         &\ll_{k,S_n,G,\epsilon} \left(\sum_{|\fq| > Y}|\fq|^{\delta+\epsilon}\right) X^{\frac{1}{|G|}}.
    \end{split}
\end{equation*}
As $\delta < -1$, we can select $\epsilon > 0$ to ensure convergence of the series $\sum_{|\fq| > Y}|\fq|^{\delta+\epsilon}$ for all $Y > 0$. Consequently, we infer that
\[
 N_{k}(S_n\times G; X)-N_{k, Y}(S_n\times G; X) \ll_{k,S_n, G, \epsilon} \left(\sum_{|\fq| > Y}|\fq|^{\delta+\epsilon}\right)X^{\frac{1}{|G|}}.
\]
for all $Y > 0$. Since the sum $\sum_{|\fq| > Y}|\fq|^{\delta+\epsilon}$ approaches zero as $Y \to \infty$ for the chosen $\epsilon > 0$, we deduce that
\[
\lim_{Y \to \infty}\limsup_{X \to \infty}\frac{N_k(S_n \times G; X)- N_{k,Y}(S_n \times G; X)}{X^{\frac{1}{|G|}}} = 0.
\]
Using \eqref{final eq} we conclude that
\[
\lim_{X\to\infty}\frac{N_k(S_n \times G; X)}{X^{\frac{1}{|G|}}}=c(k, S_n\times G).
\]
This completes the proof of Theorem \ref{main theorem}.
\end{proof}

\bibliographystyle{alpha}
\bibliography{references}
\end{document}